\documentclass[a4paper,10pt,twoside]{article}

\usepackage{tensor}
\usepackage[utf8]{inputenc}
\usepackage[UKenglish]{babel}
\usepackage{cite}
\usepackage{amssymb}
\usepackage{amsmath}
\usepackage{mathrsfs}
\usepackage{amsthm}
\usepackage[T1]{fontenc}
\usepackage{accents}
\usepackage[cm]{fullpage}
\usepackage{graphicx}
\usepackage{tensor}
\usepackage{relsize}
\usepackage{appendix}
\usepackage[affil-it]{authblk}

\usepackage{mathtools}
\mathtoolsset{showonlyrefs,showmanualtags}

\RequirePackage{lastpage}
\RequirePackage{latexsym}
\makeatletter
\ifx\@NODS\undefined\RequirePackage{dsfont}\fi
\ifx\@NOAMS\undefined\RequirePackage{amsmath,amsfonts,amssymb,amsthm}\fi
\makeatother

\RequirePackage{geometry}
\RequirePackage{bera}
\RequirePackage[pdftex,pagebackref=false]{hyperref}
\hypersetup{pdfborder=0 0 0}
\hypersetup{pdfstartview={FitH}}

\title{{\Large \bfseries{Brownian bridges to submanifolds}}}

\author{James Thompson%
  \footnote{University of Luxembourg, Email: \texttt{james.thompson@uni.lu}}}

\date{\today}

\makeatletter
\def\@MRExtract#1 #2!{#1}
\newcommand{\MR}[1]{
  \xdef\@MRSTRIP{\@MRExtract#1 !}%
  \href{http://www.ams.org/mathscinet-getitem?mr=\@MRSTRIP}{MR-\@MRSTRIP}}
\makeatother

\makeatletter
\renewenvironment{thebibliography}[1]{%
  \section*{\refname
    \@mkboth{\MakeUppercase\refname}{\MakeUppercase\refname}}%
  \phantomsection%
  \addcontentsline{toc}{section}{\refname}%
  \list{\@biblabel{\@arabic\c@enumiv}}%
  {\settowidth\labelwidth{\@biblabel{#1}}%
    \small%
    \setlength{\labelsep}{0.4em}%
    \setlength{\leftmargin}{\labelwidth}%
    \addtolength{\leftmargin}{\labelsep}%
    \setlength{\itemsep}{-.25em}%
    \@openbib@code
    \usecounter{enumiv}%
    \let\p@enumiv\@empty
    \renewcommand\theenumiv{\@arabic\c@enumiv}}%
  \sloppy\clubpenalty4000\@clubpenalty\clubpenalty\widowpenalty4000%
  \sfcode`\.\@m}{%
  \def\@noitemerr{%
    \@latex@warning{Empty `thebibliography' environment}}%
  \endlist}
\makeatother

\makeatletter
\ifx\@NODS\undefined%

\let\mathbb=\mathds
\else%
\fi
\makeatother

\DeclareMathOperator{\vol}{vol}
\DeclareMathOperator{\Ric}{Ric}

\DeclareMathOperator{\Cut}{Cut}

\DeclareMathOperator{\Hess}{Hess}
\DeclareMathOperator{\BesselI}{BesselI}
\DeclareMathOperator{\sech}{sech}
\DeclareMathOperator{\expN}{exp}

\newcommand{\dimm}{m}
\newcommand{\dimn}{n}

\newcommand{\mom}{p}
\newcommand{\secur}{\kappa}
\newcommand{\brname}{Fermi }
\newcommand{\gausscon}{\sigma^2}
\newcommand{\gcnpr}{c}

\newtheorem{theorem}{Theorem}[section]
\newtheorem{lemma}[theorem]{Lemma}
\newtheorem{proposition}[theorem]{Proposition}
\newtheorem{corollary}[theorem]{Corollary}
\newtheorem{definition}[theorem]{Definition}

\newtheorem{example}[theorem]{Example}

\begin{document}

\maketitle

\begin{abstract}%
\noindent%
   {We introduce and study Brownian bridges to submanifolds. Our method involves proving a general formula for the integral over a submanifold of the minimal heat kernel on a complete Riemannian manifold. We use the formula to derive lower bounds, an asymptotic relation and derivative estimates. We also see a connection to hypersurface local time. This work is motivated by the desire to extend the analysis of path and loop spaces to measures on paths which terminate on a submanifold.}\\[1em]%
   {\footnotesize%
     \par%
     \par
   }
 \end{abstract}

\section*{Introduction}

Brownian motion is a stochastic process naturally associated to any Riemannian manifold. The distance between Brownian motion and a submanifold was studied by the author in \cite{PAPERONE}, using an inequality for the Laplacian of the distance function based on Jacobi field comparison. The Brownian bridge is given by conditioning Brownian motion to hit a fixed point at a fixed positive time. We extend this concept by replacing the fixed point with a submanifold. There has so far been very little research on this topic, even if the ambient space is Euclidean, although such processes have appeared in the context of Wiener measure approximation, as in \cite{MR1803450}.

Suppose that $M$ is a complete and connected Riemannian manifold of dimension $\dimm$ and that $N$ is a closed embedded submanifold of $M$ of dimension $\dimn \in \lbrace 0,\ldots,\dimm-1\rbrace$. Assume for this introduction that $M$ is compact. Denote by $p^M$ the heat kernel on $M$ and define
\begin{equation}\label{eq:ihkintro}
p^M_t(x,N) := \int_N p^M_t(x,y)\,d\vol_N(y)
\end{equation}
where $\vol_N$ denotes the measure on $N$ induced by the embedding. For the case in which $N$ is a point, $\vol_N$ is a delta measure and the integral reduces to a pointwise evaluation. The integrated heat kernel will be the central object of study. Some examples of it are given in Section \ref{sec:localtime}. Following \cite{PAPERONE}, it can be identified as the time derivative of a mean hypersurface local time. A comparison theorem for this object is included in Subsection \ref{ss:hltcomp}.

In Section \ref{sec:bbridges}, denoting by $X(x)$ a Brownian motion on $M$ starting at $x$, we fix $T>0$ and prove that if $t \in \left[0,T\right)$ then for a bounded $\mathcal{F}_t^{X(x)}$-measurable random variable $F$ we have
\begin{equation}
\mathbb{E}\left[F\vert X_T(x) \in N \right] = \frac{\mathbb{E}\left[ p^M_{T-t}(X_t(x),N)F \right]}{p^M_T(x,N)}.
\end{equation}
This gives rise to a diffusion over the time interval $\left[0,T\right)$ starting at $x$ and arriving in $N$ at time $T$, with time-dependent infinitesimal generator
\begin{equation}
\frac{1}{2}\triangle + \nabla \log p^M_{T-t}(\cdot,N).
\end{equation}
We call this a \textit{Brownian bridge to a submanifold}. To show that it is a semimartingale on $\left[0,T\right]$ we prove the gradient estimate
\begin{equation}\label{eq:gradestintro}
\| \nabla \log p^M_t(x,N)\|^2 \leq C\left(\frac{1}{t}+\frac{n}{t}\log\frac{1}{t}+\frac{d^2(x,N)}{t^2}\right)
\end{equation}
and derive a Hessian estimate as corollary. These estimates are given in Section \ref{sec:deriv} by Theorem \ref{th:derivativeestiamte} and Corollary \ref{cor:hesscor}, respectively. We prove them using the method of Stroock \cite{MR1439536}, the inductive element of which had previously been discovered by Cheng, Li and Yau \cite{MR630777}. Our estimates generalize the main theorem of Engoulatov \cite{MR2253731} and the gradient and Hessian estimates of Hsu \cite{MR1618694}, who considered only the one point case.

To prove the gradient estimate we require a suitable lower bound on $p^M_t(x,N)$. In Section \ref{sec:fbridges} we define $r_N(\cdot):= d(\cdot,N)$ and introduce the diffusion on $M$ starting at $x$ with time dependent generator
\begin{equation}
\frac{1}{2}\triangle - \frac{r_N }{T-t}\frac{\partial}{\partial r_N}
\end{equation}
where $\frac{\partial}{\partial r_N}$ denotes differentiation in the radial direction. We call such a process a \textit{\brname bridge}. It is an example of a \textit{submanifold bridge process}, by which we mean a Brownian motion with drift which arrives in a submanifold at a fixed positive time. This terminology was introduced by Ndumu \cite{MNNTHESIS}. In Section \ref{sec:Sechkform} we prove the main result of this article, Theorem \ref{th:maintheorem}, which in terms of a \brname bridge $\hat{X}(x)$ yields the formula
\begin{equation}\label{eq:hkformintro}
p^M_T(x,N) = (2\pi t)^{-\frac{(\dimm-\dimn)}{2}}\exp\left[-\frac{d^2(x,N)}{2t}\right] \lim_{t\uparrow T}\mathbb{E}\left[\exp\left[\int_0^t \frac{r_N(\hat{X}_s(x))}{T-s} \bigg(d\mathbf{A}_s+d\mathbf{L}_s\bigg)\right]\right].\\[3.0mm]
\end{equation}
Here $d\mathbf{A}$ is an absolutely continuous random measure, which takes into account the geometry of $M$ in between $N$ and the cut locus, while $d\mathbf{L}$ is a singular continuous random measure, which takes in to account the effect of the cut locus itself. If $N$ is a point, this yields a formula for the heat kernel itself.

Theorem \ref{th:maintheorem} is proved using a modification of the method of Elworthy and Truman \cite{MR729719}. Elworthy and Truman's heat kernel formula and Ndumu's extension of it to integrated heat kernels \cite{MNNTHESIS,MR1100896,MR1401071,MR2786702} typically require some invertibility of the exponential map and so cannot be applied directly to manifolds with a cut locus. The problem of extending these formulae to manifolds with a cut locus was posed to the author by Xue-Mei Li. Theorem \ref{th:maintheorem} provides this extension and can also be adapted to include a drift and a potential, in analogy with the Schr\"{o}dinger equation for a magnetic field, considered by Watling \cite{KDWTHESIS,MR960168,MR1182603}.

The formulae in \cite{MR729719}, \cite{MNNTHESIS} and \cite{KDWTHESIS} are all expressed in terms of a stochastic process called the \textit{semiclassical bridge}. The semiclassical bridge will be used to obtain derivative formulae for kernels of Schr\"{o}dinger operators in a future publication by the author and Xue-Mei Li. It is similar to the \brname bridge but has an additional drift which depends on the Jacobian determinant of the exponential map. This additional drift results in the appearance of a potential function which is difficult to understand geometrically. The random measure $d\mathbf{A}$ can, on the other hand, be written explicitly in terms of Jacobi fields. 

This allows us in Section \ref{sec:lbar} to use Theorem \ref{th:maintheorem} and the estimate on $\mathbf{A}$ given in \cite{PAPERONE} to deduce a lower bound on the integrated heat kernel. This bound is stated in Theorem \ref{th:quadcurvlb}. It is used to prove an asymptotic relation, stated in Theorem \ref{th:genvar}, and in Section \ref{sec:deriv} it is used to prove the desired gradient and Hessian estimates. Potential applications of these results are discussed in Subsection \ref{ss:furtherapps}.

Although for this introduction we have assumed that $M$ is compact, this is not always necessary. In fact, we will usually only require some control on the Laplacian of the distance function. Geometric conditions for this were given in \cite{PAPERONE}. In particular, it was proved there that if there exists a function $\secur : \left[0,\infty\right) \rightarrow \left[0,\infty\right)$ such that one of the following conditions is satisfied off the union of $N$ and its cut locus:
\begin{description}
\item[(C1)] $\dimn \in \lbrace 0,\ldots,\dimm-1\rbrace$, the sectional curvature of planes containing the radial direction is bounded below by $-\kappa^2(r_N)$ and the absolute value of the principal curvature of $N$ is bounded by a constant $\Lambda \geq0$;
\item[(C2)] $\dimn=0$ and the Ricci curvature in the radial direction is bounded below by $-(\dimm-1)\kappa^2(r_N)$;
\item[(C3)] $\dimn=\dimm-1$, the Ricci curvature in the radial direction is bounded below by $-(\dimm-1)\kappa^2(r_N)$ and the absolute value of the mean curvature of $N$ is bounded by a constant $\Lambda\geq 0$;
\end{description}
then the inequality
\begin{equation}
\frac{1}{2}\triangle r_N^2 \leq (\dimm-\dimn) + \left(\dimn \Lambda + (\dimm-1)\kappa(r_N)\right)r_N
\end{equation}
holds off the cut locus pointwise and, by \cite{MR1894148}, on the whole of $M$ in the sense of distributions. Conditions of this type are used to deduce Theorem \ref{th:quadcurvlb}. 

The problem which originally motivated this piece of work was to extend Elworthy and Truman's formula to manifolds with a cut locus. Theorem \ref{th:maintheorem} was the first new result to be proved in this article and led to the remainder of the investigation. The author wishes to thank Xue-Mei Li for for advising this and the study of point to submanifold bridges, as either conditioned Brownian motions or as Brownian motions with directed drift, and for suggesting several useful references.

\section{The Heat Kernel Integrated over a Submanifold}\label{sec:localtime}

Suppose that $M$ is a complete and connected Riemannian manifold of dimension $\dimm$ and that $X(x)$ is a Brownian motion on $M$ starting at $x\in M$ defined up to an explosion time $\zeta(x)$. An open connected subset $D$ of $M$ is called a \textit{regular domain} if it has smooth boundary and compact closure. If $p^D$ denotes the Dirichlet heat kernel on a regular domain $D$ then $p^D$ is the fundamental solution to the heat equation on $D$ with Dirichlet boundary conditions. If $f$ is a non-negative measurable function on $D$ then $p^D$ satisfies
\begin{equation}
\mathbb{E}\left[ \mathbf{1}_{\lbrace t<\tau_D\rbrace} f(X_t(x))\right] = \int_M f(y)p^D_t(x,y)\,d\vol_M(y)
\end{equation}
for all $t>0$ where $\tau_D$ denotes the first exit time of $X(x)$ from $D$. If $\lbrace D_i\rbrace_{i=1}^{\infty}$ is an exhaustion of $M$ by regular domains then the minimal heat kernel on $M$ is given by $p^M:= \lim_{i\uparrow \infty}p^{D_i}$. It is the minimal fundamental solution of the heat equation on $M$ and coincides with the transition densities of Brownian motion. 

\begin{example}
On $\mathbb{R}^\dimm$ the heat kernel is given by the Gauss-Weierstrass kernel
\begin{equation}\label{eq:euclideanhk}
p_t^{\mathbb{R}^\dimm} (x,y) = {\left( 2 \pi t \right) }^{-\frac{\dimm}{2}} \exp \left(-\frac{d^2(x,y)}{2t} \right)
\end{equation}
for $x,y \in \mathbb{R}^\dimm$ and $t>0$.
\end{example}

\begin{example}
On $\mathbb{H}^{3}_\secur$, the hyperbolic space of dimension $3$ with constant sectional curvatures $\secur<0$, there is the formula
\begin{equation}\label{eq:hkonhyp3}
p^{\mathbb{H}^{3}_\secur}_t(x,y) = (2\pi t)^{-\frac{3}{2}}\exp\left[-\frac{d^2(x,y)}{2t}\right]\underbrace{\frac{\sqrt{-\secur}d(x,y)e^{\frac{\secur t}{2}}}{\sinh \left(\sqrt{-\secur}d(x,y)\right)}}_{\leq 1}
\end{equation}
for $x,y \in \mathbb{H}^{3}_\secur$ and $t>0$.
\end{example}

\begin{example}
On $\mathbb{S}^1$, the unit circle, there is the summation formula
\begin{equation}
p^{\mathbb{S}^1}_t(x,y) = (2\pi t)^{-\frac{1}{2}}\exp\left[-\frac{d^2(x,y)}{2t}\right]\underbrace{\sum_{k \in \mathbb{Z}} \exp\left[-\frac{2\pi k (d(x,y)+\pi k)}{t}\right]}_{\geq 1}
\end{equation}
for $x,y \in \mathbb{S}^1$ and $t>0$, a result of the fact that $\mathbb{S}^1$ is covered by $\mathbb{R}$.
\end{example}

The articles \cite{MR1642767} and \cite{MR2752561} provide iterative formulae for the heat kernels on the standard hyperbolic and spherical spaces, respectively, of arbitrary dimension. It is generally difficult to find deterministic formulae, unless the manifold in question exhibits a certain amount of symmetry, as in the above examples.

Now suppose that $N$ is a closed embedded submanifold of $M$ of dimension $\dimn \in \lbrace 0,\ldots,\dimm-1\rbrace$ and recall the definition of the integrated heat kernel given by formula \eqref{eq:ihkintro}. If $D$ is a regular domain in $M$ then define similarly
\begin{equation}\label{eq:intdirhk}
p^D_t(x,N):= \int_N p^D_t(x,y)d\vol_N(y)
\end{equation}
for $x\in M$ and $t>0$, where $\vol_N$ denotes the measure on $N$ induced by the embedding.

\begin{example}
For $r>0$ denote by $\mathbb{S}^1(r)$ the circle of radius $r$ centred at the origin in $\mathbb{R}^2$. Then
\begin{equation}\label{eq:circinplane}
p^{\mathbb{R}^2}_t(x,\mathbb{S}^1(r))=rt^{-1}\exp\left[-\frac{\left(r^2+\|x\|^2\right)}{2t}\right]\BesselI\left(0,\frac{r\|x\|}{t}\right)
\end{equation}
for $x \in \mathbb{R}^2$ and $t>0$, where $\BesselI$ denotes the modified Bessel function of the first kind. A plot of the integrated kernel is given below in Figure \ref{fig:ring}.
\begin{figure}[th]
\centering
\includegraphics[width=0.4\textwidth]{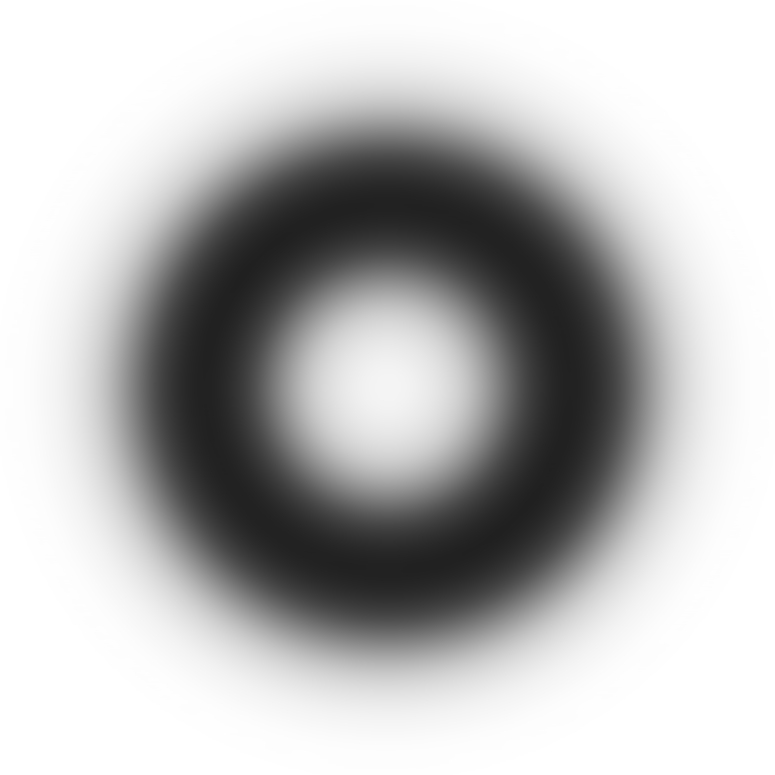}
\caption[The integral of the heat kernel on $\mathbb{R}^2$ over a circle]{\small A density plot of the integral over a circle of the heat kernel on $\mathbb{R}^2$, at a fixed small time $t>0$. The origin is located at the center of the image.}
\label{fig:ring}
\end{figure}
\normalsize
\end{example}

For $y \in M$ one can think of $p^M_{\cdot}(\cdot,y)$ as a solution to the heat equation on $M$ for the measure-valued initial condition given by the Dirac measure based at $y$. Similarly, $p^M_{\cdot}(\cdot,N)$ can be thought of as a solution to the heat equation on $M$ for the measure-valued initial condition $\vol_N$. For example, if $N$ is a closed embedded surface in $\mathbb{R}^3$ uniformly heated at time zero then $p^{\mathbb{R}^3}_{\cdot}(\cdot,N)$ describes how the heat diffuses for positive times. Alternatively, if $N$ is a closed embedded loop in $\mathbb{R}^3$ then $p^{\mathbb{R}^3}_{\cdot}(\cdot,N)$ could be used to describe the temperature around a hot metal wire. A probabilistic interpretation of the integrated heat kernel will be given in the next section by Theorem \ref{th:bridgecons}.

Our basic approach will be to compare $p^M_{\cdot}(\cdot,N)$ to the function
\begin{equation}\label{eq:qdefn}
q_t(x,N) := (2\pi t)^{-\frac{(\dimm-\dimn)}{2}} \exp\left[-\frac{d^2(x,N)}{2t}\right]
\end{equation}
since if $\mathbb{R}^\dimn$ is viewed as an affine linear subspace of $\mathbb{R}^\dimm$ then
\begin{equation}\label{eq:inthkrn}
p^{\mathbb{R}^\dimm}_T(x,\mathbb{R}^\dimn) = q_T(x,{\mathbb{R}^\dimn}).
\end{equation}

\begin{example}
Viewing the $2$-dimensional hyperbolic space $\mathbb{H}^2$ as an embedded totally geodesic submanifold of $\mathbb{H}^3$, Ndumu \cite{MNNTHESIS} used formula \eqref{eq:hkonhyp3} with $\secur =-1$ to show that
\begin{equation}\label{eq:hkformh2h3}
p^{\mathbb{H}^3}_t(x,\mathbb{H}^2) = q_t(x,\mathbb{H}^2)\frac{e^{-\frac{t}{2}}}{\cosh (d(x,\mathbb{H}^2)))}
\end{equation}
for $x \in \mathbb{H}^3$ and $t>0$.
\end{example}

\section{Brownian Bridges}\label{sec:bbridges}

Suppose in this section that $M$ is a stochastically complete and that $N$ is compact. Fix $T>0$ and $x \in M$ and consider the associated canonical probability space
\begin{equation}
(W(M),\mathcal{B}(W(M)),\mathbb{P}^{x})
\end{equation}
equipped with canonical filtration $\lbrace \mathcal{B}_t(W(M))\rbrace_{0\leq t \leq T}$. Here $W(M)$ is the space of continuous paths defined on $\left[0,T\right]$ taking values in $M$, $\mathscr{B}_t(W(M))$ is the $\sigma$-algebra generated by the coordinate maps upto time $t$ and $\mathbb{P}^{x}$ is Wiener measure, with respect to which the coordinate process $\lbrace X_t : t \in \left[0,T\right]\rbrace$ is a Brownian motion on $M$ starting at $x$.

\subsection{Conditioning on the Distance Function}

Recall that the \textit{normal exponential map} is simply the exponential map restricted to the normal bundle $TN^\perp$. Denote its Jacobian determinant by $\theta_N$. Inequalities for this object were provided in \cite{PAPERONE}. The following theorem shows how the integrated heat kernel appears naturally when a Brownian motion is conditioned to arrive in $N$ at the fixed time $T$.

\begin{theorem}\label{th:bridgecons}
Choose $t \in \left[0,T\right)$ and a bounded $\mathscr{B}_t(W(M))$-measurable random variable $F$. Then
\begin{equation}\label{eq:condondist}
\mathbb{E}^x\left[F(X)\vert X_T \in N \right] = \frac{ \mathbb{E}^{x}\left[ p^M_{T-t}(X_t,N)F(X) \right]}{p^M_T(x,N)}.
\end{equation}
\end{theorem}

\begin{proof}
For sufficiently small $\epsilon >0$ it follows from the definition of conditional expectation, the Markov property, Fubini's theorem and the smooth coarea formula that
\begin{equation}
\begin{split}
&\text{ } \mathbb{E}^x\left[F(X)\vert d(X_T,N) < \epsilon \right]\\
=&\text{ } \frac{\mathbb{E}^{x}\left[\mathbf{1}_{\lbrace d(X_T,N)< \epsilon\rbrace}F(X)\right]}{\mathbb{P}^{x}\lbrace d(X_T,N)< \epsilon\rbrace}\\
=&\text{ } \frac{\mathbb{E}^{x} \left[P_{T-t}\mathbf{1}_{\lbrace d(\cdot,N)< \epsilon\rbrace}(X_t) F(X) \right]}{\mathbb{P}^{x}\lbrace d(X_T,N)< \epsilon\rbrace}\\
=&\text{ } \frac{\int_{{B}_{\epsilon}(N)}\mathbb{E}^{x}\left[ p^M_{T-t}(X_t,y)F(X) \right]\,d\vol_M(y)}{\int_{{B}_{\epsilon}(N)} p^M_T(x,y)\,d\vol_M(y)}\\
=&\text{ } \frac{\int_N \int_{{B}_\epsilon^z(0)}\mathbb{E}^{x}\left[ p^M_{T-t}(X_t,\expN(\xi))F(X) \right]\theta_N(\xi)\,d\xi\, d\vol_N(z)}{\int_N \int_{{B}_\epsilon^z(0)} p^M_T(x,\expN(\xi))\theta_N(\xi)\,d\xi\, d\vol_N(z)}
\end{split}
\end{equation}
where ${B}_\epsilon^z(0)$ denotes the open ball in $T_z N^{\perp}$ of radius $\epsilon$ centred at the origin. Since the volume of these balls is constant and independent of $z$, it follows by the continuity of the above integrands and $\theta_N\vert_N = 1$ that
\begin{equation}
\lim_{\epsilon \downarrow 0} \mathbb{E}^x\left[F(X)\vert d(X_T,N) < \epsilon \right]=\frac{\int_N \mathbb{E}^{x}\left[ p^M_{T-t}(X_t,z)F(X) \right]\,d\vol_N(z)}{\int_N p^M_T(x,z)\,d\vol_N(z)}
\end{equation}
from which the result follows, by the definition of the left-hand side of \eqref{eq:condondist} as a Radon-Nikodym derivative.
\end{proof}

For each $t \in \left[0,T\right)$ it follows, by Theorem \ref{th:bridgecons} and the smooth coarea formula, that conditioning a Brownian motion to be in the \textit{interior} of a tubular neighbourhood of $N$ of radius $\epsilon$ at time $T$ while separately conditioning Brownian motion to belong to the \textit{boundary} of that tubular neighbourhood at time $T$ results in two measures on $\mathcal{B}_t(W(M))$ which converge weakly to the same limit as $\epsilon \downarrow 0$.

\subsection{Existence of the Bridge Measure}\label{ss:subsectgbb}

If for $y \in M$ we define a measure $\mathbb{P}^{x,y;T}$ on $\mathcal{B}(W(M))$ by $\mathbb{P}^{x,y;T}\lbrace A \rbrace=\mathbb{P}^x\lbrace A\vert X_T = y\rbrace$, for $A \in \mathcal{B}(W(M))$, then Theorem \ref{th:bridgecons} implies that $\mathbb{P}^{x,y;T}$ is absolutely continuous with respect to $\mathbb{P}^{x}$ on $\mathcal{B}_t(W(M))$ for any $t\in \left[0,T\right)$ and that the Radon-Nikodym derivative is given by
\begin{equation}
\frac{d\mathbb{P}^{x,y;T}|_{\mathcal{B}_t(W(M))}}{d\mathbb{P}^{x}} = \frac{p^M_{T-t}(X_t,y)}{p^M_T(x,y)}.
\end{equation}
In particular, if $\mathbb{P}^{x,y;T}$ exists as a probability measure on the space of continuous paths starting at $x$ and terminating at $y$ at time $T$ then under $\mathbb{P}^{x,y;T}$ and for $0<t_1<\cdots<t_k<T$ the joint density function of $X_{t_1},\ldots,X_{t_k}$, denoted by $p^M_{t_1,\ldots,t_k}(x,x_1,\ldots,x_k,y)$, is given by
\begin{equation}\label{eq:jointdens}
p^M_{t_1,\ldots,t_k}(x,x_1,\ldots,x_k,y)= \frac{p^M_{t_1}(x,x_1)p^M_{t_2-t_1}(x_1,x_2)\cdots p^M_{T-t_k}(x_k,y)}{p^M_T(x,y)},
\end{equation}
as is well-known. We will prove the existence of $\mathbb{P}^{x,y;T}$ assuming that there exist constants $c,\sigma^2>0$ such that
\begin{equation}\label{eq:assone}
p^M_t(w,z) \leq c t^{-\frac{\dimm}{2}}\exp\left[-\frac{d^2(w,z)}{\sigma^2 t}\right]
\end{equation}
for all $w,z \in M$ and $t \in \left(0,T\right]$ and a constant $\beta\geq 0$ such that
\begin{equation}\label{eq:asstwo}
\theta_w(\xi) \leq \exp\left[\beta(1+\|\xi\|^2)\right]
\end{equation}
for all $w \in M$ and $\xi \in T_wM$. By \cite{MR2388658} and \cite{PAPERONE}, such bounds exist if the Ricci curvature is bounded below by a constant with $\vol_M$ lower regular. In particular, by \cite{MR882426} and \cite{PAPERONE}, such bounds exist if the Ricci curvature is bounded with $M$ having positive injectivity radius, in which case the constants $c$ and $\sigma^2$ can be chosen so that $\sigma^2$ is arbitrarily close to $2$. For the case in which $M$ is compact, a version of the following lemma was proved by Driver in \cite{MR1154540}.

\begin{lemma}\label{lem:concrit}
Assume \eqref{eq:assone} and \eqref{eq:asstwo}, suppose $0\leq s<t\leq T$ and without loss of generality assume $\beta>0$. Then for all $\gamma \in (0,1)$ there exists a constant $C(\dimm,c,\sigma^2,\beta,\gamma,T)>0$ such that for all $p>0$ we have
\begin{equation}
\mathbb{E}^{x,y;T}\left[d^{\mom}(X_s,X_t)\right] \leq \frac{C(\dimm,c,\sigma^2,\beta,\gamma,T)}{p^M_T(x,y)}\frac{\Gamma\left(\frac{\dimm+\mom}{2}\right)}{\Gamma\left(\frac{\dimm}{2}\right)}\left(\frac{\sigma^2(t-s)}{1-\gamma}\right)^{\frac{\mom}{2}}
\end{equation}
so long as $t-s<\gamma(\sigma^2\beta)^{-1}$.
\end{lemma}

\begin{proof}
First assume $0< s<t<2T/3$. If $a>0$ then
\begin{equation}
\int_0^\infty \exp\left[-ar^2\right]r^{\dimm-1}ds = \frac{\Gamma(\dimm/2)}{2a^{\frac{\dimm}{2}}}
\end{equation}
and so for $w \in M$ we have
\begingroup
\allowdisplaybreaks
\begin{align}\nonumber
& \text{ } \int_M p^M_{t-s}(w,z)d^{\mom}(w,z)\,d\vol_M(z)\\ \nonumber
\leq & \text{ } c (t-s)^{-\frac{\dimm}{2}}\int_M \exp\left[-\frac{d^2(w,z)}{\sigma^2(t-s)}\right]d^\mom(w,z)\,d\vol_M(z)\\ \nonumber
\leq & \text{ } c e^\beta (t-s)^{-\frac{\dimm}{2}}\int_{T_wM} \| \xi\|^{\mom}\exp\left[\left(\beta-\frac{1}{\sigma^2(t-s)}\right)\| \xi\|^2\right]dv\\ \nonumber
=& \text{ } c e^\beta \frac{\dimm\pi^{\frac{\dimm}{2}}(t-s)^{-\frac{\dimm}{2}}}{\Gamma\left(\frac{\dimm}{2}+1\right)}\int_0^\infty r^{\mom+\dimm-1}\exp\left[\left(\beta-\frac{1}{\sigma^2(t-s)}\right)r^2\right]dr\\ \nonumber
=& \text{ }c e^\beta\frac{ \pi^{\frac{\dimm}{2}}\Gamma\left(\frac{\dimm+\mom}{2}\right)(t-s)^{-\frac{\dimm}{2}}}{\Gamma\left(\frac{\dimm}{2}\right)}\left(\frac{1-\beta\sigma^2(t-s)}{\sigma^2(t-s)}\right)^{-\frac{(\dimm+\mom)}{2}}\\ \nonumber
\leq & \text{ } c e^\beta\frac{\Gamma\left(\frac{\dimm+\mom}{2}\right)}{\Gamma\left(\frac{\dimm}{2}\right)}\left(\frac{\sigma^2\pi}{1-\gamma}\right)^{\frac{\dimm}{2}}\left(\frac{\sigma^2(t-s)}{1-\gamma}\right)^{\frac{\mom}{2}}.\nonumber
\end{align}%
\endgroup
Thus there exists a constant $C_0(\dimm,c,\sigma^2,\beta,\gamma)>0$ such that
\begin{equation}
\int_M p^M_{t-s}(w,z)d^{\mom}(w,z)d\vol_M(z) \leq C_0(\dimm,c,\sigma^2,\beta,\gamma) \frac{\Gamma\left(\frac{\dimm+\mom}{2}\right)}{\Gamma\left(\frac{\dimm}{2}\right)}\left(\frac{\sigma^2(t-s)}{1-\gamma}\right)^{\frac{\mom}{2}} \nonumber
\end{equation}
for all $\mom>0$, $w \in M$ and $s,t$ satisfying $t-s<\gamma(\sigma^2\beta)^{-1}$. For such $s,t$ we see that
\begin{equation}
\begin{split}
 &\text{ }\mathbb{E}^{x,y;T}\left[d^{\mom}(X_s,X_t)\right]\\
= &\text{ }\int_M \int_M \frac{p^M_s(x,w)p^M_{t-s}(w,z)d^{\mom}(w,z) p^M_{T-t}(z,y)}{p^M_T(x,y)} d\vol_M(w)d\vol_M(z) \\
\leq &\text{ }(T/3)^{-\dimm/2}\frac{ C_1(\dimm,c,\sigma^2,\beta,\gamma)}{p^M_T(x,y)}  \frac{\Gamma\left(\frac{\dimm+\mom}{2}\right)}{\Gamma\left(\frac{\dimm}{2}\right)}\left(\frac{\sigma^2(t-s)}{1-\gamma}\right)^{\frac{\mom}{2}}. \nonumber
\end{split}
\end{equation}
The same result is obtained for $T/3<s<t<T$. The cases $s=0$ or $t=T$ can be treated similarly.
\end{proof}

If $F$ is a bounded $\mathscr{B}_t(W(M))$-measurable function on $W(M)$ for some $t \in \left[0,T\right)$ then it follows from Theorem \ref{th:bridgecons} that
\begin{equation}\label{eq:gentobas}
\mathbb{E}^{x}\left[F(X)\vert X_T \in N \right] = \frac{\int_N p^M_T(x,y)\mathbb{E}^{x,y;T}\left[ F(X) \right]\,d\vol_N(y)}{p^M_T(x,N)}
\end{equation}
which implies, by Lemma \ref{lem:concrit}, that for all $p\geq2$ there exists $\epsilon,C_\epsilon>0$ such that
\begin{equation}
\mathbb{E}^{x}\left[d^{\mom}(X_s,X_t)\vert X_T \in N \right]  \leq C_\epsilon(t-s)^{\frac{\mom}{2}}\nonumber
\end{equation}
for all $0\leq s<t\leq T$ with $t-s<\epsilon$. It follows from Kolmogorov's continuity theorem, by covering the interval $\left[0,T\right]$ with finitely many closed intervals each of length less that $\epsilon$, that there exists a probability measure $\mathbb{P}^{x;N,T}$ on the bridge space
\begin{equation}
L_{x;N,T}(M) := \lbrace \omega \in W(M):X_0(\omega)=x, \text{ }X_T(\omega) \in N\rbrace \nonumber
\end{equation}
which satisfies $\mathbb{P}^{x;N,T}\lbrace A \rbrace= \mathbb{P}^{x}\lbrace A\vert X_T \in N\rbrace$ for $A \in \mathscr{B}(W(M))$ and whose finite-dimensional distributions can be deduced from equations \eqref{eq:jointdens} and \eqref{eq:gentobas}. In particular, if $\mathbb{P}^{x;N,T}_{X_T}$ denotes the law of the random variable $X_T$ under the measure $\mathbb{P}^{x;N,T}$ then
\begin{equation}
\mathbb{P}^{x;N,T}_{X_T} = \frac{p^M_T(x,\cdot)}{p^M_T(x,N)}\vol_N.
\end{equation}
For example, if $M=\mathbb{R}^\dimm$ with $N$ the unit $(\dimm-1)$-sphere and $x=0$ then $\mathbb{P}^{x;N,T}_{X_T}$ is given by the uniform measure on $N$. In contrast, if $M=\mathbb{R}^\dimm$ with $N$ an $\dimn$-dimensional subspace and $x=0$ then $\mathbb{P}^{x;N,T}_{X_T}$ is given by the heat kernel measure on $N$. Note that the measure $\mathbb{P}^{x;N,T}$ clearly exists in the latter case, even though $N$ is non-compact.

Using Lemma \ref{lem:concrit} we can prove the following concentration inequality for tubular neighbourhoods.

\begin{proposition}\label{prop:concencondthm}
Assuming \eqref{eq:assone} and \eqref{eq:asstwo}, for all $\gamma \in (0,1)$ there exists $\epsilon>0$ such that
\begin{equation}
\lim_{r\uparrow \infty} \frac{1}{r^2}\log \mathbb{P}^{x;N,T}\lbrace X_t \not\in B_r(N)\rbrace \leq -\frac{1-\gamma}{\sigma^2(T-t)}
\end{equation}
for all $t\in \left(T-\epsilon,T\right]$.
\end{proposition}

\begin{proof}
Without loss of generality, assume $\beta >0$. Since $d(\cdot,N)= \inf_{y \in N}d(\cdot,y)$ it follows that for each $y \in N$ we have $\mathbb{E}^{x,y;T}\left[d^p(X_t,N)\right] \leq \mathbb{E}^{x,y;T}\left[d^p(X_t,y)\right]$. Therefore, by applying \eqref{eq:gentobas} with $F(X)=d^p(X_t,N)$, we see by Lemma \ref{lem:concrit} that for all $\gamma \in (0,1)$ there exists $C(\dimm,c,\sigma^2,\beta,\gamma,N,T)>0$ such that
\begin{equation}
\mathbb{E}^{x;N,T}[d^{\mom}(X_t,N)]\leq  \frac{C(\dimm,c,\sigma^2,\beta,\gamma,N,T)}{p^M_T(x,N)}\frac{\Gamma\left(\frac{\dimm+\mom}{2}\right)}{\Gamma\left(\frac{\dimm}{2}\right)}\left(\frac{\sigma^2(T-t)}{1-\gamma}\right)^{\frac{\mom}{2}}\nonumber
\end{equation}
for all $0\leq t\leq T$ with $T-t<\gamma(\sigma^2\beta)^{-1}$. For such $t$, choosing $\theta>0$, applying this bound to the case where $\mom$ is an even integer and summing yields
\begin{equation}
\mathbb{E}^{x;N,T}\left[e^{\frac{\theta}{2} d^2(X_t,N)}\right] \leq C(\dimm,c,\sigma^2,\beta,\gamma,x,N,T) \left(1- \frac{\theta\sigma^2(T-t)}{2(1-\gamma)}\right)^{-\frac{\dimm}{2}} \nonumber
\end{equation}
so long as $t>T-2(1-\gamma)(\theta\sigma^2)^{-1}$. Under these conditions on $t$, it follows from Markov's inequality that for all $r>0$ there is the estimate
\begin{equation}
\mathbb{P}^{x;N,T}\lbrace X_t \not\in B_r(N)\rbrace \leq C(\dimm,\delta,\gamma,c_R,x,N,T)\left(1- \frac{\theta\sigma^2(T-t)}{2(1-\gamma)}\right)^{-\frac{\dimm}{2}}e^{-\frac{\theta r^2}{2}}.\nonumber
\end{equation}
Fixing $\delta \in \left[0,1\right)$ and choosing $\theta = 2\delta(1-\gamma)(\sigma^2(T-t))^{-1}$ this yields
\begin{equation}
\lim_{r\uparrow \infty} \frac{1}{r^2}\log \mathbb{P}^{x;N,T}\lbrace X_t \not\in B_r(N)\rbrace \leq -\frac{\delta(1-\gamma)}{\sigma^2(T-t)}\nonumber
\end{equation}
from which the result follows since $\delta$ can be chosen arbitrarily close to $1$.
\end{proof}

\subsection{Semimartingale Property}\label{ss:smprop}

It follows from \eqref{eq:gentobas} and Girsanov's theorem that under the measure $\mathbb{P}^{x;N,T}$ the coordinate process on $W(M)$ is a diffusion on the half-open time interval $\left[0,T\right)$ starting at $x$ with time-dependent infinitesimal generator
\begin{equation}
\frac{1}{2}\triangle + \nabla \log p^M_{T-t}(\cdot,N)
\end{equation}
for $t\in\left[0,T\right)$. To show that it is a semimartingale under this measure over the closed time interval $\left[0,T\right]$ we require a suitable estimate on the logarithmic derivative. This estimate will be deduced from the results obtained in Sections \ref{sec:fbridges}, \ref{sec:Sechkform} and \ref{sec:lbar}. It is given by Theorem \ref{th:derivativeestiamte} in Section \ref{sec:deriv} and the semimartingale property is given by Corollary \ref{cor:hesscor}.

\section{Fermi Bridges}\label{sec:fbridges}

In this section we introduce a bridge process defined in terms of the distance function $r_N(\cdot):=d(\cdot,N)$. Recall that the cut locus of $N$, denoted by $\Cut(N)$, is given by the closure of the set of points at which $r_N^2$ fails to be differentiable. It is a closed set of $\vol_M$-measure zero which is, upto a set of Hausdorff dimension at most $\dimm-2$, an at most countable union $\mathring{C}(N)$ of open subsets of hypersurfaces. These facts follow from the work of Angulo Ardoy and Guijarro \cite{MR2951748} and Mantegazza and Mennucci \cite{MR1941909} who used Hamilton-Jacobi equations and the theory of viscosity solutions. The vector field $\frac{\partial}{\partial r_N}$ will denote differentiation in the radial direction, defined off the union of $N$ and $\Cut(N)$ to be the gradient of $r_N$ and set equal to zero elsewhere. For $T>0$ fixed with $q_{\cdot}\left(\cdot,N\right)$ defined by equation \eqref{eq:qdefn}, the time-dependent vector field
\begin{equation}
\mathbf{1}_{M\setminus \Cut(M)} \nabla \log q_{T-t}(\cdot,N) = - \frac{r_N }{T-t}\frac{\partial}{\partial r_N} \nonumber
\end{equation}
where $t\in\left[0,T\right)$ is smooth away from the cut locus but generally not continuous on it. One imagines the deterministic flow associated to this vector field as being one for which $\Cut(N)$ is a source and for which $N$ is a sink. The strength of the flow increases dramatically as the terminal time $T$ is approached, while the vector field vanishes on $N$. A diffusion on $M$ starting at $x$ with time-dependent infinitesimal generator
\begin{equation}
\frac{1}{2}\triangle - \frac{r_N }{T-t}\frac{\partial}{\partial r_N}
\end{equation}
will be called a \textit{\brname bridge between $x$ and $N$ in time $T$}. Such processes will always be defined upto a predictable stopping time less than or equal to $T$. We use the name \textit{\brname bridge} since the time-dependent part of the drift acts in a direction normal to $N$, which would be the radial part of a system of polar Fermi coordinates, and since there are conditions under which this process arrives at $N$ at time $T$ almost surely. Note that if $M=\mathbb{R}^\dimm$ with $N$ a point then the above definition reduces to that of a standard Brownian bridge.

\subsection{Radial Part}\label{ss:firstsec}

Suppose that $\hat{X}(x)$ is a \brname bridge between $x$ and $N$ in time $T$, defined upto the minimum of $T$ and its explosion time. Suppose also that $D$ is a regular domain in $M$ and denote by $\hat{\tau}_D$ the first exit time of $\hat{X}(x)$ from $D$. Since $\check{C}(N)$ is polar for $\hat{X}(x)$ and since the martingale part of an antidevelopment of $\hat{X}(x)$ is a standard Brownian motion, the It\^{o} formula of Barden and Le \cite[Thoerem~1]{MR1314177} implies that there exist continuous adapted non-decreasing and non-negative processes $L^{N}(\hat{X}(x))$ and $L^{\mathring{C}(N)}(\hat{X}(x))$ whose associated random measures are singular with respect to Lebesgue measure and supported when $\hat{X}(x)$ takes values in $N$ and $\mathring{C}(N)$, respectively, such that
\begin{equation}\label{eq:rNFormulaforbridge}
\begin{split}
{r_N}(\hat{X}_{t \wedge \hat{\tau}_D}(x)) &= r_N(x) + \beta_{t \wedge \hat{\tau}_D} +\frac{1}{2}\int_0^{t \wedge \hat{\tau}_D} \triangle r_N(\hat{X}_s(x))\,ds\\[3.0mm]
&\text{ }\quad\quad -\int_0^{t \wedge \hat{\tau}_D} \frac{r_N(\hat{X}_s(x))}{T-s}\,ds-\mathbb{L}^{\Cut(N)}_{t \wedge \hat{\tau}_D}(\hat{X}(x)) + L^N_{t \wedge \hat{\tau}_D}(\hat{X}(x))
\end{split}
\end{equation}
for all $t\in \left[0,T\right)$, almost surely, where $\beta$ is a standard one-dimensional Brownian motion and
\begin{equation}\label{eq:mathbbLdefnforbr}
d\mathbb{L}^{\Cut(N)}(\hat{X}(x)) :=-\frac{1}{2}\left(D^+- D^-\right)_{\hat{X}(x)} {r_N} (\mathbf{n}) \, dL^{\mathring{C}(N)}(\hat{X}(x)).
\end{equation}
Here $\mathbf{n}$ is any unit normal vector field on $\mathring{C}(N)$ and the G\^{a}teaux derivatives $D^\pm {r_N}$ are defined for $z \in \mathring{C}(N)$ and $v \in T_zM$ by
\begin{equation}
D^+_z {r_N}(v) := \lim_{\epsilon \downarrow 0} \frac{1}{\epsilon} \left(f(\exp_z (\epsilon v)) - {r_N}(z)\right) \nonumber
\end{equation}
and $D^-_z {r_N}(v) := -D^+_z {r_N}(-v)$. The processes $L^N(\hat{X}(x))$ and $L^{\mathring{C}(N)}(\hat{X}(x))$ are, roughly speaking, given by the local times at zero of $d(\hat{X}(x),N)$ and $d(\hat{X}(x),\mathring{C}(N))$, respectively, whenever the latter make sense. Using this formula we can estimate the radial moments of the \brname bridge, restricting our attention to the domain $D$ and doing so under the assumption that there exist constants $\nu \geq1$ and $\lambda \geq 0$ such that
\begin{equation}\label{eq:triineq}
\frac{1}{2}\triangle r_N^2 \leq \nu + \lambda r_N^2
\end{equation}
on $D \setminus \Cut(N)$.

\begin{theorem}\label{th:secradmomthmbr}
Let $\nu \geq1$ and $\lambda \geq 0$ be any constants such that inequality \eqref{eq:triineq} holds on $D\setminus \Cut(N)$. Then we have
\begin{equation}\label{eq:secondradmomforbridge}
\mathbb{E} \left[\mathbf{1}_{\lbrace t<\hat{\tau}_D\rbrace} r_N^2(\hat{X}_t(x))\right] \leq \left(r_N^2 (x) \left( \frac{T-t}{T}\right) + \nu t \right) \left( \frac{T-t}{T}\right) e^{\lambda t}
\end{equation}
for all $t\in \left[0,T\right)$.
\end{theorem}

\begin{proof}
Define the function $\hat{f}_{x,2}:\left[0,T\right)\rightarrow \mathbb{R}$ by
\begin{equation}
\hat{f}_{x,2}(t) := \mathbb{E} [\mathbf{1}_{\lbrace t<\hat{\tau}_D\rbrace} r^2_N (\hat{X}_t(x))] \nonumber
\end{equation}
for $t \in \left[0,T\right)$. By It\^{o}'s formula and formula \eqref{eq:rNFormulaforbridge} we deduce the differential inequality
\begin{equation}
\begin{cases}
\hat{f}'_{x,2} (t) \leq \nu  + \left(\lambda-\frac{2}{T-t}\right) \hat{f}_{x,2} (t)\\
\hat{f}_{x,2}(0) = r_N^2(x) \nonumber
\end{cases}
\end{equation}
for all $t \in \left[0,T\right)$. Applying Gronwall's inequality to it yields
\begin{equation}
\begin{split}
\hat{f}_{x,2} (t) &\leq \left(r_N^2(x)+\nu \int_0^t \left(\frac{T}{T-s}\right)^2 e^{-s\lambda} ds\right)\left(\frac{T-t}{t}\right)^2 e^{\lambda t}\\
&\leq \left(r_N^2(x)+\nu t \left(\frac{t}{T-t}\right)\right)\left(\frac{T-t}{t}\right)^2 e^{\lambda t}\nonumber
\end{split}
\end{equation}
where we used the assumption $\lambda \geq 0$ for the second inequality.
\end{proof}

Note that for $M=\mathbb{R}^\dimm$ with $N$ a linear subspace, with $\nu=\dimm-\dimn$ and $\lambda = 0$, one can set $D=M$ and inequality \eqref{eq:secondradmomforbridge} holds as an \textit{equality}. Also, Jensen's inequality implies the following estimate on the first radial moment.

\begin{corollary}\label{cor:frmebr}
Let $\nu \geq1$ and $\lambda \geq 0$ be any constants such that inequality \eqref{eq:triineq} holds on $D\setminus \Cut(N)$. Then we have
\begin{equation}
\mathbb{E} [\mathbf{1}_{\lbrace t<\hat{\tau}_D\rbrace}r_N (\hat{X}_t(x))] \leq \left(\left(r_N^2 (x) \left( \frac{T-t}{T}\right) + \nu t\right) \left( \frac{T-t}{T}\right)\right)^{\frac{1}{2}} e^{\frac{\lambda t}{2}}
\end{equation}
for all $t\in \left[0,T\right)$.
\end{corollary}

In the next subsection we consider the case where there exist constants $\nu \geq1$ and $\lambda \geq 0$ such that inequality \eqref{eq:triineq} holds on the whole of $M\setminus \Cut(N)$.

\subsection{Bridge Property}\label{ss:lapassump}

Suppose for this subsection that there exists constants $\nu \geq1$ and $\lambda \geq 0$ such that inequality \eqref{eq:triineq} holds on $M\setminus \Cut(N)$. Suppose also that $X(x)$ is a \textit{non-explosive} Brownian motion on $M$ starting at $x$, defined on a filtered probability space
\begin{equation}
\left(\Omega,\mathcal{F},\lbrace \mathcal{F}_t \rbrace_{t\geq 0},\mathbb{P}\right)
\end{equation}
satisfying the usual conditions. Note that, by \cite[Theorem~5]{PAPERONE}, if $N$ is compact then inequality \eqref{eq:triineq} \textit{implies} the non-explosion of $X(x)$. For $t \in \left[0,T\right)$ define
\begin{equation}
M_t :=\exp \left[ - \int_0^t \frac{r_N(X_s(x))}{T-s}\bigg \langle\frac{\partial}{\partial r_N},U_s  dB_s\bigg\rangle - \frac{1}{2} \int_0^t \frac{r_N^2(X_s(x))}{(T-s)^2}ds\right] \nonumber
\end{equation}
where $U$ is a horizontal lift of $X(x)$ to the orthonormal frame bundle with antidevelopment $B$. It follows from \cite[Proposition~3]{PAPERONE} and Novikov's criterion that $M$ is a martingale up to time $t$ for each $t \in \left[0,T\right)$. For each $t \in \left[0,T\right)$ we can therefore define a new probability measure $\mathbb{Q}_t$ on $\mathcal{F}_t$ by $d\mathbb{Q}_t = M_t d\mathbb{P}$. It follows from Girsanov's theorem that the process $X(x)$ when restricted to $\left[0,t\right)$ and considered on the filtered probability space
\begin{equation}
\left(\Omega,\mathcal{F}_t,\lbrace \mathcal{F}_s \rbrace_{s\in \left[0,t\right)},\mathbb{Q}_t\right)
\end{equation}
is a \brname bridge between $x$ and $N$ in time $T$. We therefore obtain a new process, denoted by $\hat{X}(x)$, defined on $\left[0,T\right)$ and equivalent to the Brownian motion $X(x)$ on $\left[0,t\right)$ for each $t\in \left[0,T\right)$. By considering an exhaustion of $M$ by regular domains, Theorem \ref{th:secradmomthmbr} and the monotone convergence theorem imply that it satisfies the bridge property
\begin{equation}
\lim_{t \uparrow T}r_N(\hat{X}_t(x)) =0
\end{equation}
almost surely. In particular, if $N$ is a point $p$ then one then can extend $\hat{X}(x)$ to a continuous process on $\left[0,T\right]$ by setting $\hat{X}_T(x)= p$.

\section{Heat Kernel Formula and Comparison Theorem}\label{sec:Sechkform}

In this section we prove the main result of this article. Denote by $\mathcal{M}(N)$ the largest domain in $TN^\bot$ with star-like fibres and such that $\expN \vert_{\mathcal{M}(N)}$ is a diffeomorphism onto its image. Then the image of $\expN \vert_{\mathcal{M}(N)}$ is $M\setminus \Cut(N)$. Recalling that $\theta_N$ denotes the Jacobian determinant of the normal exponential, define also
\begin{equation}\label{eq:Thetadefn}
\Theta_N :=\theta_N \circ {\left(\expN \vert_{\mathcal{M}(N)}\right)}^{-1}
\end{equation}
and note that from \cite{MR2024928} there is the formula
\begin{equation}\label{eq:regE}
\frac{1}{2}\triangle r_N^2 = (\dimm-\dimn) +r_N \frac{\partial}{\partial r_N} \log \Theta_N
\end{equation}
on $M \setminus \Cut(N)$. Recall also the definition of the geometric local time $\mathbb{L}^{\Cut(N)}(\hat{X}(x))$ given in Subsection \ref{ss:firstsec} by formula \eqref{eq:mathbbLdefnforbr}.

\begin{theorem}\label{th:maintheorem}
Suppose that $M$ is a complete and connected Riemannian manifold of dimension $\dimm$, that $N$ is a closed embedded submanifold of $M$ of dimension $\dimn\in \lbrace 0,\ldots,\dimm-1\rbrace$ and that $D$ is a regular domain in $M$. Suppose that $x\in M$ with $T>0$ and that $\hat{X}(x)$ is a \brname bridge between $x$ and $N$ in time $T$. Denote by $\hat{\tau}_D$ the first exit time of this process from $D$. Then, with $q_T(x,N)$ and $p^D_T(x,N)$ defined by \eqref{eq:qdefn} and \eqref{eq:intdirhk}, respectively, we have
\begin{equation}
p^D_T(x,N) = q_T(x,N) \lim_{t\uparrow T}\mathbb{E}\left[\mathbf{1}_{\lbrace t<\hat{\tau}_D \rbrace}\exp\left[\int_0^t \frac{r_N(\hat{X}_s(x))}{T-s} \bigg(d\mathbf{A}_s+d\mathbf{L}_s\bigg)\right]\right]
\end{equation}
where
\begin{equation}\label{eq:ALdefns}
d\mathbf{A}_s := \frac{\partial}{\partial r_N}\log \Theta^{-\frac{1}{2}}_N(\hat{X}_s(x))\,ds, \quad d\mathbf{L}_s := d\mathbb{L}^{\Cut(N)}_s(\hat{X}(x)).
\end{equation}
\end{theorem}

\begin{proof}
To begin with, we see by the smooth coarea formula that
\small
\begin{equation}
\begin{split}
&\text{ }\lim_{t\uparrow T} \int_M p^{D}_t(x,y)q_{T-t}(y,N)d\vol_M(y)\\
=& \text{ } \lim_{t\uparrow T} \int_N \int_{T_zN^{\perp}} (p^{D}_t(x,\expN)\mathbf{1}_{\mathcal{M}_z(N)}\theta_N)(\xi)(2\pi (T-t))^{-\frac{(\dimm-\dimn)}{2}}\exp\left[-\frac{\| \xi\|^2}{2(T-t)}\right]d\xi d\vol_N(z)\\
=& \text{ } \lim_{t\uparrow T} \int_N \int_{T_zN^{\perp}} (p^{D}_t(x,\expN)\mathbf{1}_{\mathcal{M}_z(N)}\theta_N)(\sqrt{T-t}\xi)(2\pi)^{-\frac{(\dimm-\dimn)}{2}}\exp\left[-\frac{\|\xi\|^2}{2}\right]d\xi d\vol_N(z)\\
=& \text{ } \int_N \int_{T_zN^{\perp}} (p^{D}_T(x,\expN)\mathbf{1}_{\mathcal{M}_z(N)}\theta_N)(0_z)(2\pi)^{-\frac{(\dimm-\dimn)}{2}}\exp\left[-\frac{\|\xi\|^2}{2}\right]d\xi d\vol_N(z)\\
=&\text{ } \int_N p^{D}_T (x,z)d\vol_N(z)
\end{split}
\end{equation}
\normalsize
where $0_z$ denotes the origin of the vector space $T_z N^\perp$ and where the third equality follows from the compactness of the closure of $D$ and the dominated convergence theorem. Then, denoting by $\lbrace P^{D}_t :t\geq0 \rbrace$ the Dirichlet heat semigroup, it follows from Girsanov's theorem that
\begin{equation}
\begin{split}
\int_N p^{D}_T(x,z)d\vol_N(z) &= \lim_{t\uparrow T} \int_M p^{D}_t(x,y)q_{T-t}(y,N)d\vol_M(y) \\
&= \lim_{t\uparrow T} P^{D}_t q_{T-t}(\cdot,N)(x)\\
&=\lim_{t\uparrow T} \mathbb{E}\left[\mathbf{1}_{\lbrace t<\hat{\tau}_D\rbrace} q_{T-t}(\hat{X}_t(x),N)\hat{M}_t\right] \nonumber
\end{split}
\end{equation}
with
\begin{equation}\label{eq:girscov}
\hat{M}_{t\wedge \hat{\tau}_D}= \exp \left[ \int_0^{t\wedge \hat{\tau}_D} \frac{r_N(\hat{X}_s(x))}{T-s}\bigg \langle\frac{\partial}{\partial r_N},\hat{U}_s  dB_s \bigg\rangle- \frac{1}{2} \int_0^{t\wedge \hat{\tau}_D} \frac{r_N^2(\hat{X}_s(x))}{(T-s)^2}\,ds\right]
\end{equation}
where $\hat{U}$ is a horizontal lift of $\hat{X}(x)$ to the orthonormal frame bundle and where $B$ is the associated $\mathbb{R}^\dimm$-valued Brownian motion given by the antidevelopment of $\hat{U}$. Now, using It\^{o}'s formula and formula \eqref{eq:rNFormulaforbridge} we deduce that
\begin{equation}
\begin{split}
&\text{ }\log q_{T-({t\wedge \hat{\tau}_D})} (\hat{X}_{t\wedge \hat{\tau}_D}(x),N) \\
=&\text{ }\log q_T (x,N) - \int_0^{t\wedge \hat{\tau}_D} \frac{r_N(\hat{X}_s(x))}{T-s}\bigg \langle\frac{\partial}{\partial r_N},\hat{U}_s  dB_s\bigg\rangle \\
& \text{ }+\int_0^{t\wedge \hat{\tau}_D} \frac{\partial}{\partial s} \log q_{T-s} (\hat{X}_s(x),N) ds+\int_0^{t\wedge \hat{\tau}_D}  \frac{r_N^2(\hat{X}_s(x))}{(T-s)^2}ds  ds \\
& \text{ }+\frac{1}{2}\int_0^{t\wedge \hat{\tau}_D}\triangle \log q_{T-s} (\hat{X}_s(x),N) ds\\
& \text{ } + \int_0^{t\wedge \hat{\tau}_D} \frac{r_N(\hat{X}_s(x))}{T-s}d\mathbb{L}^{\Cut(N)}_s(\hat{X}(x)). \nonumber
\end{split}
\end{equation}
and so we can eliminate the stochastic integral in \eqref{eq:girscov} by rearrangement and substitution. Finally, using the fact that
\begin{equation}
\frac{\partial}{\partial s} \log q_{T-s}(\cdot,N) = \frac{\dimm-\dimn}{2(T-s)} - \frac{r_N^2(\cdot)}{2(T-s)^2} \nonumber
\end{equation}
and also that
\begin{equation}
\triangle \log q_{T-s} (\cdot,N) = -\frac{\triangle r_N^2(\cdot)}{2(T-s)} \nonumber
\end{equation}
on $M\setminus \Cut(N)$ together with formula \eqref{eq:regE} we can further simplify the resulting expression so as to obtain the desired formula.
\end{proof}

\begin{theorem}\label{th:ihkform}
Suppose that $\lbrace D_i \rbrace_{i=1}^{\infty}$ is an exhaustion of $M$ by regular domains. Then, with $p^M_T(x,N)$ defined by \eqref{eq:ihkintro}, we have
\begin{equation}\label{eq:limitform}
p^M_T(x,N) = q_T(x,N)\lim_{i\uparrow \infty} \lim_{t\uparrow T}\mathbb{E}\left[\mathbf{1}_{\lbrace t<\hat{\tau}_{D_i}\rbrace}\exp\left[\int_0^t \frac{r_N(\hat{X}_s(x))}{T-s} \bigg(d\mathbf{A}_s+d\mathbf{L}_s\bigg)\right]\right]
\end{equation}
where $d\mathbf{A}$ and $d\mathbf{L}$ are defined by \eqref{eq:ALdefns}.
\end{theorem}
\begin{proof}
Recalling that $p^M$ is given as the limit of the increasing sequence of Dirichlet heat kernels $p^{D_i}$, it follows from the monotone convergence theorem that
\begin{equation}
p^M_T(x,N) = \lim_{i\uparrow \infty} p^{D_i}_T(x,N)
\end{equation}
and so the result follows from Theorem \ref{th:maintheorem}.
\end{proof}

The absolutely continuous random measure $d\mathbf{A}$ can be understood precisely in terms of Jacobi fields, using Heintze and Karcher's comparison theorem for $\frac{\partial}{\partial r_N}\log \theta_N$ which was carefully explained in \cite[Section~3]{MR533065}. The singular random measure $d\mathbf{L}$ vanishes if the cut locus is polar for $\hat{X}(x)$. In particular, we have the following corollary of Theorem \ref{th:ihkform}, which shows that we recover identity \eqref{eq:inthkrn} in the Euclidean setting.

\begin{corollary}\label{cor:yarasas}
Suppose that $M$ is stochastically complete, that the cut locus of $N$ has Hausdorff dimension at most $\dimm-2$ and that one of the following conditions is satisfied:
\begin{description}
\item[(A1)] $\dimn \in \lbrace 0,\ldots,\dimm-1\rbrace$, the sectional curvature of planes containing the radial direction vanishes and $N$ is totally geodesic;
\item[(A2)] $\dimn \in \lbrace 0,\dimm-1\rbrace$, the Ricci curvature in the radial direction vanishes and $N$ is minimal.
\end{description}
Then we have
\begin{equation}
p^M_T(x,N) = q_T(x,N)\nonumber
\end{equation}
for all $x \in M$ and $T>0$.
\end{corollary}
\begin{proof}
The corollary follows from Theorem \ref{th:ihkform}, the Heintze-Karcher theorem and the monotone convergence theorem.
\end{proof}

If $N$ is a point then Theorem \ref{th:ihkform} provides a formula for the heat kernel itself. In this case, by the Heintze-Karcher theorem, the random measure $d\mathbf{A}$ has the property of being non-negative if the Ricci curvature of $M$ is non-negative and non-positive if the Ricci curvature of $M$ is non-positive. The Heintze-Karcher also yields a comparison theorem, stated below, in which we view $\mathbb{H}^\dimn_\kappa$ as a totally geodesic embedded submanifold of $\mathbb{H}^\dimm_\kappa$ and set $\mathbb{H}^\dimn_0 := \mathbb{R}^\dimn$ and $\mathbb{H}^\dimm_0 := \mathbb{R}^\dimm$. It is a partial generalization of the heat kernel comparison theorem of Cheeger and Yau \cite{MR615626}.

\begin{theorem}\label{th:ihkcomparison}
Suppose that $M$ is stochastically complete and that one of the following conditions is satisfied:
\begin{description}
\item[(B1)] $\dimn \in \lbrace 0,\ldots,\dimm-1\rbrace$, the sectional curvature of planes containing the radial direction is bounded below by $\kappa\leq 0$ and $N$ is totally geodesic;
\item[(B2)] $\dimn \in \lbrace 0,\dimm-1\rbrace$, the Ricci curvature in the radial direction is bounded below by $(\dimm-1)\kappa\leq 0$ and $N$ is minimal.
\end{description}
Then we have
\begin{equation}
p^M_T(x,N)\geq p^{\mathbb{H}^\dimm_\kappa}_T(y,{\mathbb{H}^\dimn_\kappa})
\end{equation}
for any $y \in \mathbb{H}^\dimm_\kappa$ satisfying $r_N(x)\leq r_{\mathbb{H}^\dimn_\kappa}(y)$.
\end{theorem}

\begin{proof}
First suppose $\kappa =0$. Then $\frac{\partial}{\partial r_N}\log \Theta_N \leq 0$, by \cite[Theorem~1]{PAPERONE}, so the assertion follows from Theorem \ref{th:ihkform} since $\mathbf{L}$ is non-decreasing. So assume $\kappa <0$ and define a function $g_\kappa$ by
\begin{equation}
g_\kappa(t) := (\dimm-\dimn-1)\left(\sqrt{-\kappa}t \coth (\sqrt{-\kappa}t)-1\right)+\dimn\sqrt{-\kappa}t \tanh(\sqrt{-\kappa}t ).
\end{equation}
Then, by the Heintze-Karcher theorem, it follows that
\begin{equation}
r_N\frac{\partial}{\partial r_N}\log \Theta_N \leq g_\kappa(r_N).
\end{equation}
Since $\mathbf{L}$ is non-decreasing with $g_\kappa$ non-negative, it follows from this, Theorem \ref{th:ihkform} and dominated convergence that
\begin{equation}
p^M_T(x,N) \geq q_T(x,N)\mathbb{E}\left[\exp\left[-\int_0^T \frac{g_\kappa(r_N(\hat{X}_s(x)))}{2(T-s)} \,ds\right]\right].
\end{equation}
Now let $\hat{Y}(y)$ be a \brname bridge between $y$ and $\mathbb{H}^\dimn_\kappa$ in time $T$ such that its radial part $r_{\mathbb{H}^\dimn_\kappa}(\hat{Y}(y))$ satisfies a stochastic differential equation driven by the one-dimensional Brownian motion $\beta$ which drives the equation for $r_{N}(\hat{X}(x))$. Then the Heintze-Karcher theorem, applied to the coefficients of the stochastic differential equation for $r_N^2(\hat{X}(x))$, combined with the comparison theorem for solutions to stochastic differential equations given in \cite{MR2199554}, implies that $r_N(\hat{X}(x))\leq r_{\mathbb{H}^\dimn_\kappa}(\hat{Y}(y))$, almost surely. Since $g_\kappa$ is non-decreasing this yields
\begin{equation}\label{eq:ihkformmodel}
p^M_T(x,N) \geq q_T(y,\mathbb{H}^\dimn_\kappa)\mathbb{E}\left[\exp\left[-\int_0^T \frac{g_\kappa(r_{\mathbb{H}^\dimn_\kappa}(\hat{Y}_s(y)))}{2(T-s)} \,ds\right]\right].
\end{equation}
But the right-hand side of \eqref{eq:ihkformmodel} is equal to $p^{\mathbb{H}^\dimm_\kappa}_T(y,{\mathbb{H}^\dimn_\kappa})$, by Theorem \ref{th:ihkform} and the fact that
\begin{equation}
g_\kappa(r_{\mathbb{H}^\dimn_\kappa}) = r_{\mathbb{H}^\dimn_\kappa}\frac{\partial}{\partial r_{\mathbb{H}^\dimn_\kappa}}\log \Theta_{\mathbb{H}^\dimn_\kappa},
\end{equation}
so the result follows.
\end{proof}

\subsection{Hypersurface Local Time}\label{ss:hltcomp}

If $X(x)$ is a Brownian motion starting at $x \in M$ then $r_N(X(x))$ is a continuous semimartingale. For the case in which $N$ is a hypersurface, the local time of this process at zero is denoted by $L^N(X(x))$ and referred to as the local time of $X(x)$ on $N$. In certain circumstances, such as when $M$ is compact, it was explained in \cite{PAPERONE} that there is the formula
\begin{equation}\label{eq:revuzform}
\mathbb{E}\left[L^N_t(X(x))\right] = \int_0^t p^M_s(x,N)\,ds.
\end{equation}
In the hypersurface case, the integrated heat kernel is therefore given by the time derivative of a mean local time.

\begin{example}
Viewing the $2$-dimensional hyperbolic space $\mathbb{H}^2$ as an embedded totally geodesic submanifold of $\mathbb{H}^3$, if $X(x)$ is a Brownian motion in $\mathbb{H}^3$ starting at $x$ then $r_{\mathbb{H}^2}(X(x))$ is a Markov process and so, by results in \cite{MR1278079} and a change of variables, formula \eqref{eq:revuzform} holds in this setting. Therefore, by \eqref{eq:hkformh2h3}, it follows that
\begin{equation}
\begin{split}
\lim_{t\uparrow \infty}\mathbb{E}\left[ L^{\mathbb{H}^2}_t(X(x))\right] &= \sech (r_{\mathbb{H}^2}(x))) \int_0^\infty (2\pi t)^{-\frac{1}{2}} \exp\left[-\frac{r_{\mathbb{H}^2}^2(x)}{2t}-\frac{t}{2}\right] dt\\
&= \sech(r_{\mathbb{H}^2}(x))\exp\left[-r_{\mathbb{H}^2}(x)\right].
\end{split}
\end{equation}
In contrast, if $\mathbb{R}^2$ is viewed as a linear subspace of $\mathbb{R}^3$ with $X(x)$ a Brownian motion in $\mathbb{R}^3$ starting at $x$ then
\begin{equation}
\lim_{t\uparrow \infty}\mathbb{E}\left[ L^{\mathbb{R}^2}_t(X(x))\right] = \infty.
\end{equation}
\end{example}

Theorem \ref{th:ihkcomparison} yields the following comparison for the mean local time, where we retain the notational convention outlined before Theorem \ref{th:ihkcomparison}.

\begin{corollary}
Suppose that $N$ is a minimal hypersurface and that the Ricci curvature in the radial direction is bounded below by $(\dimm-1)\kappa\leq 0$. Suppose also that $X(x)$ is a non-explosive Brownian motion on $M$ starting at $x$, denote by $L^N(X(x)$ the local time on $X(x)$ on $N$ and assume that formula \eqref{eq:revuzform} holds. Denote by $Y(y)$ a Brownian motion on $\mathbb{H}^\dimm_\kappa$ starting at $y$ with $r_N(x) \leq r_{\mathbb{H}^{\dimm-1}_\kappa}(y)$. Then
\begin{equation}\label{eq:ltcomp}
\mathbb{E}[ L^N_t(X(x))] \geq \mathbb{E}[ L^{\mathbb{H}^{\dimn-1}_\kappa}_t(Y(y))]
\end{equation}
for all $t\geq0$.
\end{corollary}

\begin{proof}
This follows from directly from Theorem \ref{th:ihkcomparison} and formula \eqref{eq:revuzform}.
\end{proof}

In particular, if $\kappa = 0$ with $x \in N$ and $r_{\mathbb{H}_0^{\dimm-1}}(y)=0$ then the right-hand side of \eqref{eq:ltcomp} is equal to $\sqrt{\frac{2t}{\pi}}$.

\section{Lower Bound and Asymptotic Relation}\label{sec:lbar}

The heat kernel lower bounds of Cheeger and Yau \cite{MR615626} were proved using a bound on the Ricci curvature in the radial direction and a Laplacian comparison theorem. These are similar to the objects we use, although our method is quite different. Our method is closer in spirit to that of Wang \cite{MR1487443}, who also used stochastic techniques with unbounded curvature, but only in the one point case. Note that the constants $C_1,C_2$ and $\Lambda$ appearing in the following theorem typically depend upon $N$.

\begin{theorem}\label{th:quadcurvlb}
Suppose that $M$ is stochastically complete and that there exist constants $C_1,C_2 \geq 0$ such that one of the following conditions is satisfied off the union of $N$ and its cut locus:
\begin{description}
\item[(C1)] the sectional curvatures of planes containing the radial direction are bounded below by $-(C_1 + C_2 r_N)^2$ and there exists a constant $\Lambda \geq 0$ such that the principal curvatures of $N$ are bounded in modulus by $\Lambda$;
\item[(C2)] $\dimn=0$ and the Ricci curvature in the radial direction is bounded below by $-(\dimm-1)(C_1 + C_2 r_N)^2$;
\item[(C3)] $\dimn=\dimm-1$ and the Ricci curvature in the radial direction is bounded below by $-(\dimm-1)(C_1 + C_2 r_N)^2$ and there exists a constant $\Lambda\geq 0$ such that the mean curvature of $N$ is bounded in modulus by $\Lambda$.
\end{description}
Then for each $T>0$ there exists a constant $C\geq 0$, depending only on $T,C_1,C_2,\Lambda,\dimm$ and $\dimn$, such that
\begin{equation}
p^M_t(x,N) \geq t^{-\frac{(\dimm-\dimn)}{2}}\exp\left[-\frac{r_N^2(x)}{2t}-C(1+r^2_N(x))\right]\nonumber
\end{equation}
for all $x \in M$ and $t \in \left(0,T\right]$.
\end{theorem}

\begin{proof}
By \cite[Theorem~1]{PAPERONE}, the curvature conditions imply
\begin{equation}\label{eq:ddrest}
\frac{\partial}{\partial r_N}\log \Theta_N \leq \alpha + \beta r_N
\end{equation}
with $\alpha:= \dimn \Lambda + (\dimm-1)C_1$ and $\beta:= (\dimm-1)C_2$. Using the process $\hat{X}(x)$ constructed in Subsection \ref{ss:lapassump} it follows from this, Theorem \ref{th:ihkform} and the fact that $\mathbb{L}^{\Cut(N)}(\hat{X}(x))$ is non-decreasing that
\begin{equation}
p^M_T(x,N)\geq q_T(x,N) \lim_{i\uparrow \infty} \lim_{t\uparrow T} \mathbb{E}\left[\mathbf{1}_{\lbrace t < \hat{\tau}_{D_i}\rbrace}\exp\left[-\int_0^t \frac{f(\hat{X}_s(x))}{T-s}ds\right]\right]
\end{equation}
where $f := \frac{1}{2}\left(\alpha r_N+\beta r_N^2\right)$. For $t \in \left[0,T\right)$ we see that
\begin{equation}
\begin{split}
\mathbf{1}_{\lbrace t < \hat{\tau}_{D_i}\rbrace}\exp\left[-\int_0^t\frac{f(\hat{X}_s(x))}{T-s}ds\right]&= \mathbf{1}_{\lbrace t < \hat{\tau}_{D_i}\rbrace} \sum_{\mom=0}^{\infty} \frac{{\left(-\int_0^t \frac{f(\hat{X}_s(x))}{T-s}ds\right)^\mom}}{\mom!}\\
&= \mathbf{1}_{\lbrace t < \hat{\tau}_{D_i}\rbrace} -1 + \sum_{\mom=0}^{\infty} \frac{\left(-\mathbf{1}_{\lbrace t < \hat{\tau}_{D_i}\rbrace}\int_0^t \frac{f(\hat{X}_s(x))}{T-s}ds\right)^\mom}{\mom!}\\
&= \mathbf{1}_{\lbrace t < \hat{\tau}_{D_i}\rbrace} -1 +\exp\left[-\mathbf{1}_{\lbrace t < \hat{\tau}_{D_i}\rbrace}\int_0^t \frac{f(\hat{X}_s(x))}{T-s}ds\right]
\end{split}
\end{equation}
from which it follows, by Jensen's inequality, that
\begin{equation}\label{eq:lbineqpf}
\begin{split}
& \quad\text{ }\mathbb{E}\left[\mathbf{1}_{\lbrace t < \hat{\tau}_{D_i}\rbrace}\exp\left[-\int_0^t\frac{f(\hat{X}_s(x))}{T-s}ds\right]\right]\\
&\quad\quad\quad\quad\quad \geq \mathbb{Q}_t \lbrace t < \hat{\tau}_{D_i}\rbrace-1 + \exp\left[-\mathbb{E}\left[\mathbf{1}_{\lbrace t < \hat{\tau}_{D_i}\rbrace}\int_0^t\frac{f(\hat{X}_s(x))}{T-s}ds \right]\right].
\end{split}
\end{equation}
Inequality \eqref{eq:ddrest} implies by formula \eqref{eq:regE} that inequality \eqref{eq:triineq} holds on $M\setminus \Cut(N)$ with $\nu = \dimm-\dimn+\frac{\alpha}{2}$ and $\lambda = \frac{\alpha}{2}+\beta$ and so, by Theorem \ref{th:secradmomthmbr}, we see that
\begin{equation}
\begin{split}
\mathbb{E} \left[\mathbf{1}_{\lbrace t<\hat{\tau}_{D_i}\rbrace}\int_0^t \frac{r_N^2 (\hat{X}_s(x))}{T-s}ds\right] &\leq \int_0^t \frac{\mathbb{E}[\mathbf{1}_{\lbrace s<\hat{\tau}_{D_i}\rbrace}r_N^2(\hat{X}_s(x))]}{T-s}ds\\
&\leq \left(r_N^2(x)+\nu T\right)e^{\lambda T}
\end{split}
\end{equation}
for all $t \in \left[0,T\right)$. Similarly, by Corollary \ref{cor:frmebr}, we have
\begin{equation}
\begin{split}
\mathbb{E} \left[\mathbf{1}_{\lbrace t<\hat{\tau}_{D_i}\rbrace}\int_0^t \frac{r_N (\hat{X}_s(x))}{T-s}ds\right] &\leq \int_0^t \frac{\mathbb{E}[\mathbf{1}_{\lbrace s<\hat{\tau}_{D_i}\rbrace}r_N(\hat{X}_s(x))]}{T-s}ds \\
&\leq \left(\frac{r_N^2 (x) + \nu t}{T}\right)^{\frac{1}{2}}e^{\frac{\lambda t}{2}} \int_0^t \left(T-s\right)^{-\frac{1}{2}}ds  \\
&= 2 \left(r_N^2 (x) + \nu t\right)^{\frac{1}{2}} e^{\frac{\lambda t}{2}} \left(1- \left(\frac{T-t}{T}\right)^{\frac{1}{2}}\right)\\
&\leq 2 \left(r_N^2 (x) + \nu T\right)^{\frac{1}{2}} e^{\frac{\lambda T}{2}}
\end{split}
\end{equation}
for all $t \in \left[0,T\right)$. Consequently
\begin{equation}
\mathbb{E}\left[\mathbf{1}_{\lbrace t < \hat{\tau}_{D_i}\rbrace}\int_0^t\frac{f(\hat{X}_s)}{T-s}ds \right]\leq \alpha \left(r_N^2(x)+\nu T\right)^{\frac{1}{2}}e^{\frac{\lambda T}{2}} + \frac{\beta}{2}\left(r_N^2(x)+\nu T\right)e^{\lambda T}
\end{equation}
for all $t \in \left[0,T\right)$.
Furthermore
\begin{equation}
\lim_{i\uparrow \infty} \lim_{t\uparrow T}\mathbb{Q}_t \lbrace t < \hat{\tau}_{D_i}\rbrace = 1,
\end{equation}
by the bridge property and the monotone convergence theorem, so we have a lower bound
\begin{equation}
\begin{split}
p_T^M (x,N)&\geq q_T(x,N)\exp\left[-\alpha \left(r_N^2(x)+\nu T\right)^{\frac{1}{2}}e^{\frac{\lambda T}{2}}-\frac{\beta}{2}\left(r_N^2(x)+\nu T\right)e^{\lambda T}\right]\\
&= q_T(x,N)\exp\left[-(\dimn \Lambda + (\dimm-1)C_1) \left(r_N^2(x)+\nu T\right)^{\frac{1}{2}}e^{\frac{\lambda T}{2}}\right. \\
&\quad\quad\quad\quad\quad\quad\quad\quad\quad\quad\quad\quad\quad\quad\quad \left.-\frac{(\dimm-1)C_2}{2}\left(r_N^2(x)+\nu T\right)e^{\lambda T}\right]
\end{split}
\end{equation}
for all $T>0$ and $x \in M$, from which the result follows.
\end{proof}

Since the heat kernel is a positive fundamental solution to the heat equation, fixing one of the spatial variables for small times results in densities whose mass is localized around that fixed point. Riemannian manifolds are locally Euclidean, so for these small times the resulting densities should be comparable to the Gauss-Weierstrass kernel \eqref{eq:euclideanhk}. The precise sense in which this is true is given by Varadhan's asymptotic relation, originally proved in \cite{MR0217881,MR0208191}, which states for the minimal heat kernel $p^M$ of a complete Riemannian manifold $M$ that
\begin{equation}\label{eq:varadrelpt}
\lim_{t\downarrow 0} t \log p^M_t(x,y) = -\frac{d^2(x,y)}{2}
\end{equation}
uniformly on compact subsets of $M\times M$. Similarly, the embedding of $N$ in $M$ is locally diffeomorphic to an affine embedding of $\mathbb{R}^\dimn$ in $\mathbb{R}^\dimm$ so for small times the integrated heat kernel $p^M_\cdot(\cdot,N)$ should be comparable to the kernel $q_\cdot(\cdot,N)$ defined by equation \eqref{eq:qdefn}.

\begin{theorem}\label{th:genvar}
Suppose that $M$ is a complete and connected Riemannian manifold of dimension $\dimm$ and that $N$ is a compactly embedded submanifold of $M$ of dimension $\dimn \in \lbrace 0,\ldots,\dimm-1\rbrace$. Then
\begin{equation}\label{eq:varadrel}
\lim_{t\downarrow 0} t\log p^M_t(x,N)= -\frac{d^2(x,N)}{2}
\end{equation}
uniformly on compact subsets of $M$.
\end{theorem}

\begin{proof}
It is a simple matter to show that the left-hand side of \eqref{eq:varadrel} is less than or equal to the right-hand side, using Varadhan's relation and the fact that $r_N(x)\leq r_y(x)$ for all $y \in N$. To prove the opposite inequality assume first that $M$ is compact. Then the result follows immediately from Theorem \ref{th:quadcurvlb}. So assume that $M$ is non-compact, let $K$ be any compact subset of $M$ and for $x \in K$ and $y \in N$ denote by $\Gamma_{x,y}$ the set of all length-minimizing geodesic segments between $x$ and $y$, viewed as a subset of $M$. Then $\Gamma_{x,y}$ contains (the image of) at least one such geodesic and by the triangle inequality the set
\begin{equation}
\Gamma_{K,N}:=\bigcup_{x\in K,y \in N} \Gamma_{x,y}
\end{equation}
is a bounded subset of $M$. Now let $D$ be any regular domain in $M$ containing $\Gamma_{K,N}$. Modify $M$ outside of $D$ so as to obtain a compact Riemannian manifold $M_D$ (by doubling, for example) and suppose that $D$ is sufficiently large so that
\begin{equation}
\lim_{t\downarrow 0} \frac{p^D_t(x,y)}{p^{M_D}_t(x,y)}=1
\end{equation}
uniformly for $x\in K$ and $y \in N$. This is the principle of \textit{not feeling the boundary} considered by Hsu in \cite{MR1325580}. Such $D$ can always be found, as explained by Norris in \cite{MR1484769}, since we are assuming that $M$ is non-compact. Then for all $\epsilon >0$ there exists $t_{\epsilon,K} >0$ such that for $t\in (0,t_{\epsilon,K})$ we have
\begin{equation}
(1-\epsilon) p^{M_D}_t(x,N) \leq p^D_t(x,N) \leq p^M_t(x,N)
\end{equation}
for all $x \in K$. It follows from this and the result in the compact case that
\begin{equation}
\lim_{t \downarrow 0} t \log p^M_t(x,N) \geq -\frac{d_{M_D}^2(x,N)}{2}
\end{equation}
where $d_{M_D}$ denotes the distance function on $M_D$. But since $\Gamma_{K,N}$ is contained in $D$ it follows that $x\in K$ and $y \in N$ implies $d_{M_D}(x,y) \leq d(x,y)$. Therefore $d_{M_D}(x,N) \leq d(x,N)$ and the result follows.
\end{proof}


\begin{example}
Suppose that $D$ is a regular domain. Then $\partial D$ is a compactly embedded hypersurface and, according to \cite[Theorem~5.2.6]{MR1882015}, one has
\begin{equation}
\lim_{t\downarrow0}t\log \mathbb{P}\lbrace \tau_D(x)<t\rbrace = -\frac{d^2(x,\partial D)}{2}
\end{equation}
for all $x \in D$, where $\tau_D(x)$ denotes the first exit time from $D$ of a Brownian motion $X(x)$ starting at $x$. On the other hand, according to \cite[Theorem~1.2]{MR1484769}, one has
\begin{equation}
\lim_{t\downarrow0}t\log p^M_t(x,D,x) = -\frac{d^2(x,\partial D)}{2}
\end{equation}
for all $x\not\in D$, where $p^M_t(x,D,y):=p^M_t(x,y)-p^{M\setminus \overline{D}}_t(x,y)$ is the measure of heat passing through $D$. Theorem \ref{th:genvar} and formula \eqref{eq:revuzform} imply that, in either case, we have
\begin{equation}
\lim_{t\downarrow0}t\log \frac{d}{dt}\mathbb{E}\left[L^{\partial D}_t(X(x))\right] = -\frac{d^2(x,\partial D)}{2}.
\end{equation}
\end{example}

\begin{example}
If $TM$ is equipped with the Sasaki metric \cite{MR0112152,MR0145456} then Theorem \ref{th:genvar} implies
\begin{equation}
\lim_{t\downarrow0}t\log p^{TM}_t(\xi,M) = -\frac{\|\xi\|^2}{2}
\end{equation}
uniformly for $\xi$ in compact subsets of $TM$.
\end{example}

One could also prove Theorem \ref{th:genvar} using Ndumu's asymptotic expansion \cite{MR2786702}. While this expansion is only valid away from the cut locus, it could be used in place of our lower bounds by connecting points in $K$ to $N$ with smooth curves, covering them with small balls and invoking the Markov property.

Alternatively, it was proved by Hino and Ram\'{i}rez in \cite{MR1988472}, in the context of Dirichlet spaces, that
\begin{equation}
\lim_{t\downarrow 0} t \log \mathbb{P}\lbrace X_0 \in A;X_t \in B\rbrace  = -\frac{d^2(A,B)}{2}
\end{equation}
for all measurable sets $A$ and $B$ of positive measure, where $X$ denotes the Markov process associated with the underlying local regular Dirichlet form and where $d$ is an associated intrinsic distance. Using Sturm's upper bounds \cite{MR1355744} and pointwise lower bounds, one can deduce from this the pointwise relation of Varadhan, as shown by Ram\'{i}rez \cite[Theorem~4.1]{MR1809739}. A modification of this approach, replacing balls with tubular neighbourhoods and the pointwise lower bounds with our integrated heat kernel lower bounds, might also be used to deduce a relation similar to \eqref{eq:varadrel}.

\subsection{Obtaining an Upper Bound from the Lower Bound}\label{oubflb}

Grigor'yan Hu and Lau showed in \cite{MR2388658} how Gaussian upper bounds on the heat kernel can be obtained from lower bounds. While lower bounds are frequently obtained from upper bounds, as in the method of Aronson \cite{MR0217444}, this was the first result to go in the other direction. Having proved a lower bound on the integrated heat kernel, and since deducing an upper bound directly from Theorem \ref{th:maintheorem} seems hard, it is natural for us to use this approach to obtain an upper bound. To do so, we require the following definition.
\begin{definition}\label{eq:lwrreg}
We say that $\vol_M$ is \textnormal{lower regular} if there exist constants $T_0,C>0$ such that
\begin{equation}
V_r(x) \geq C r^\dimm
\end{equation}
for all $x\in M$ and $0<r<\sqrt{T_0}$.
\end{definition}
For example, if the injectivity radius of $M$ is positive and the Ricci curvature is bounded above by a constant then $\vol_M$ is lower regular. For geometric conditions which imply a positive injectivity radius, see \cite{MR658471}. Now assume that $\vol_M$ is lower regular and that the Ricci curvature of $M$ is bounded below, by a constant. Then the lower bound of Theorem \ref{th:quadcurvlb} implies a \textit{near-diagonal lower estimate} of the form
\begin{equation}
p_t^M(x,y) \geq  C't^{-\frac{\dimm}{2}}
\end{equation}
for all $0<t<T_0$ and $x,y\in M$ satisfying $d(x,y)< t^\frac{1}{2}$. This implies, by \cite[Corollary~3.5]{MR2388658}, that there exist constants $\gcnpr,\gausscon>0$ (which might depend upon $T_0$) such that
\begin{equation}\label{eq:heatkernelUB}
p_t^M(x,y) \leq \gcnpr t^{-\frac{\dimm}{2}}\exp\left[-\frac{d^2(x,y)}{\gausscon t}\right]
\end{equation}
for all $t \in (0,T_0)$ and $x,y \in M$. We therefore have the following theorem, by observing that $y \in N$ implies $r_y(x)\geq r_N(x)$.
\begin{theorem}\label{th:grigub}
Suppose that the Ricci curvature of $M$ is bounded below by a constant, that $\vol_M$ is lower regular and that $N$ is compact. Then there exist constants $\gcnpr,\gausscon>0$ (which might depend upon $T_0$) such that
\begin{equation}
p^M_t(x,N)\leq \gcnpr t^{-\frac{\dimm}{2}} \exp\left[-\frac{r_N^2(x)}{\gausscon t}\right] \nonumber
\end{equation}
for all $x \in M$ and $t\in (0,T_0)$.
\end{theorem}

\section{Gradient and Hessian Estimates}\label{sec:deriv}

We are now in a position to prove gradient and Hessian estimates. They are global and hence require assumptions on curvature. For comparison, recall the derivative estimates of Cheng, Li and Yau \cite{MR630777}. They proved, in particular, that if the injectivity radius of $M$ is positive and there exist constants $A_i$ which bound the $i$\textsuperscript{th} covariant derivatives of the curvature tensor then for all $T>0$ there exist positive constants $\alpha(\dimm)$ and $C(\dimm,A_0,\ldots,A_{l-1},T)$ such that
\begin{equation}
\| \nabla^l p^M_t(\cdot,y)_x\| \leq C(\dimm,A_0,\ldots,A_{l-1},T) t^{-\frac{(\dimm+l)}{2}} \exp\left[-\frac{\alpha(\dimm)d^2(x,y)}{t}\right]
\end{equation}
for all $x,y \in M$ and $t \in \left[0,T\right]$. Our estimates will be on the \textit{logarithmic} derivatives of the \textit{integrated} heat kernel. They will be derived using a slight modification of the method used by Stroock in \cite{MR1715265}. We will show precisely how the gradient estimate is derived, whereas for the Hessian estimate we will refer to \cite{MR1715265} for the details. In particular, to prove the gradient estimate we will need Theorem \ref{th:quadcurvlb} and a formula for the derivative of the heat semigroup, which we will now state.

Assume that the Ricci curvature of $M$ is bounded below, denote by $X(x)$ a Brownian motion on $M$ starting at $x$ and denote by $U$ a horizontal lift of $X(x)$ with antidevelopment $B$. Denote also by $\lbrace \mathcal{Q}_s:s\geq 0\rbrace$ the solution the ordinary differential equation
\begin{equation}\label{eq:RicODE}
\begin{cases}
\dot{\mathcal{Q}}_s = -\frac{1}{2}\Ric_{U_s} \mathcal{Q}_s\\
\mathcal{Q}_0 = U_0^{-1}
\end{cases}
\end{equation}
where $\Ric_{U_s} := U_s^{-1} \Ric^{\sharp} U_s$. Then for any bounded measurable function $f:M\rightarrow \mathbb{R}$ there is the formula
\begin{equation}\label{eq:formder}
d(P_t f)_x = \mathbb{E}\left[f(X_t(x))\frac{1}{t}\int_0^t\langle \mathcal{Q}_s,dB_s\rangle\right]
\end{equation}
for all $t>0$. This differentiation formula was proved by Elworthy and Li in \cite{MR1297021} and by Li in \cite{XMLTHESIS}, using basic stochastic calculus, while another, related, approach was given by Thalmaier in \cite{MR1488139}. It was proved originally by Bismut in \cite{MR755001} for the case in which $M$ is compact. An analogous formula was used by Stroock, whose estimate is based upon the following lemma, proved as \cite[Lemma~6.45]{MR1715265} using Jensen's inequality.

\begin{lemma}\label{lem:stroocklem}
Suppose $(\Omega,\mathcal{F},\mathbb{P})$ is a probability space and $\phi$ a non-negative measurable function on $\Omega$ with $\mathbb{E}\left[\phi\right]=1$. If $\phi$ is a measurable function on $\Omega$ such that $\phi \psi$ is integrable, then $\mathbb{E}\left[\phi \psi \right] \leq \mathbb{E}\left[\phi \log \phi\right]+\log \mathbb{E}[e^\psi]$.
\end{lemma}

\begin{theorem}\label{th:derivativeestiamte}
Suppose that $M$ is a complete and connected Riemannian manifold of dimension $\dimm$ and that $N$ is a compactly embedded submanifold of $M$ of dimension $\dimn \in \lbrace 0, \ldots,\dimm-1\rbrace$. Suppose that the injectivity radius of $M$ is positive and that the curvature tensor is bounded. Then for all $T>0$ there exists a constant $C>0$ such that
\begin{equation}\label{eq:gradbnd}
\| \nabla \log p^M_t(\cdot,N)_x\|^2 \leq C\left(\frac{1}{t}+\frac{\dimn}{t}\log\frac{1}{t}+\frac{d^2(x,N)}{t^2}\right)
\end{equation}
for all $x \in M$ and $t \in \left(0,T\right]$.
\end{theorem}

\begin{proof}
For a bounded measurable positive function $f$ and $\gamma \in \mathbb{R}$ set
\begin{equation}
\phi := \frac{f(X_t(x))}{P_tf(x)},\quad \psi := \gamma \int_0^t \langle \mathcal{Q}_sv,dB_s\rangle
\end{equation}
and see, by Lemma \ref{lem:stroocklem} and formula \eqref{eq:formder}, that
\begin{equation}
\gamma t \frac{d(P_t f)_x}{P_tf(x)} \leq h_t(x;f)+\log \mathbb{E}\left[\exp\left[\gamma \int_0^t \langle \mathcal{Q}_s,dB_s\rangle \right]\right] \nonumber
\end{equation}
where
\begin{equation}\label{eq:htdefn}
h_t(x;f):= \mathbb{E}\left[\frac{f(X_t(x))}{P_tf(x)}\log \frac{f(X_t(x))}{P_tf(x)}\right].
\end{equation}
Furthermore, denoting by $R$ the minimum of the Ricci curvature, equation \eqref{eq:RicODE} implies
\begin{equation}
\log \mathbb{E}\left[\exp\left[\gamma \int_0^t \langle \mathcal{Q}_s v,dB_s\rangle \right]\right] \leq \frac{\gamma^2}{2}\int_0^t e^{-Rs}ds
\end{equation}
and so, after minimizing over $\gamma$, we deduce
\begin{equation}
\bigg\vert \frac{d(P_t f)_x}{P_tf(x)}\bigg\vert \leq \frac{1}{t}\left(2 h_t(x;f) \int_0^t e^{-Rs}ds\right)^{\frac{1}{2}}.\vspace{1 mm}
\end{equation}
Now we choose $f(\cdot)=p^M_t(\cdot,N)$. Then $P_tf(x) = p^M_{2t}(x,N)$, by Tonelli's theorem, and for all $z \in M$ it follows that
\begin{equation}\label{eq:asas}
h_t(x;p^M_t(\cdot,N)) \leq \sup_{z\in M} \log \left(\frac{p^M_t(z,N)}{p^M_{2t}(x,N)}\right).
\end{equation}
The curvature assumptions imply that there exist constants $C_1,C_2 \geq 0$ such that the sectional curvatures of planes containing the radial direction are bounded below by $-(C_1+C_2r_N)^2$ and so, by Theorem \ref{th:quadcurvlb}, there exists a constant $c_1\geq 0$, depending only on $T,C_1,C_2,\dimm$ and $\dimn$, such that
\begin{equation}\label{eq:lbforder}
p^M_{2t}(x,N) \geq (2t)^{-\frac{(\dimm-\dimn)}{2}}\exp\left[-\frac{r_N^2(x)}{4t}-c_1(1+r^2_N(x))\right]
\end{equation}
for all $x \in M$ and $t \in \left(0,T\right]$. The assumptions also imply, by Theorem \ref{th:grigub} and the Chapman-Kolmogorov equation, that there exists a constant $c_2>0$ such that
\begin{equation}\label{eq:ubforder}
p^M_t(z,N) \leq c_2t^{-\frac{\dimm}{2}}
\end{equation}
for all $t \in \left(0,T\right]$ and $z \in M$. Substituting the estimates \eqref{eq:lbforder} and \eqref{eq:ubforder} in to \eqref{eq:asas} proves the theorem.
\end{proof}

It follows that the gradient estimate \eqref{eq:gradbnd} holds automatically if $M$ is compact. Furthermore, we can now prove the semimartingale property promised in Subsection \ref{ss:smprop}. As pointed out by Thalmaier, the semimartingale property should hold without any assumptions on curvature. G\"{u}neysu proved this for the one point case using local estimates on the heat kernel. We do not yet have such estimates for the integrated heat kernel and so, for the time being, we make do with the assumptions of Theorem \ref{th:derivativeestiamte}.

\begin{corollary}
Under the assumptions of Theorem \ref{th:derivativeestiamte}, the coordinate process on the bridge space $L_{x;N,T}$ is a semimartingale with respect to the bridge measure $\mathbb{P}^{x;N,T}$.
\end{corollary}

\begin{proof}
It suffices to control the singularity in the drift close to the terminal time. Since the distance function $r_N$ is defined as an infimum over $N$, it follows from \eqref{eq:gentobas} and Lemma \ref{lem:concrit} that there exists $\epsilon,C_\epsilon>0$ such that $\mathbb{E}^{x;N,T}[r_N^2(X_t)] \leq C_\epsilon(T-t)$ for all $t \in \left(T-\epsilon,T\right]$. Therefore, by Theorem \ref{th:derivativeestiamte}, there exists $C>0$ such that
\begin{equation}
\begin{split}
 & \text{ }\mathbb{E}^{x;N,T}\left[\int_{T-\epsilon}^{T} \| \nabla \log p^M_{T-t}(X_t,N)\|dt\right]\\
\leq & \text{ } \int_{T-\epsilon}^{T} \mathbb{E}^{x;N,T}\left[\| \nabla \log p^M_{T-t}(X_t,N)\|^2\right]^{\frac{1}{2}} dt\\
\leq & \text{ } \sqrt{C}\int_{T-\epsilon}^{T} \left(\frac{1}{T-t}+\frac{\dimn}{T-t}\log\frac{1}{T-t}+\frac{\mathbb{E}^{x;N,T}\left[r^2_N(X_t)\right]}{(T-t)^2}\right)^{\frac{1}{2}} dt\\
\leq & \text{ } \sqrt{C}\int_{T-\epsilon}^{T} \left(\frac{1}{T-t}+\frac{\dimn}{T-t}\log\frac{1}{T-t}+\frac{C_\epsilon}{T-t}\right)^{\frac{1}{2}} dt\\
< & \text{ } \infty
\end{split}
\end{equation}
and the result follows.
\end{proof}

Note that we could have derived our gradient estimate using assumptions slightly weaker than those of Theorem \ref{th:derivativeestiamte}, using Ricci curvature, the sectional curvature of planes containing the radial direction and lower regularity of the volume measure. To derive our Hessian estimate we need only add to these assumptions some suitable control on the curvature two-form and the derivative of the Ricci tensor. For simplicity, however, we state our Hessian estimate, as we did the gradient estimate, in terms of injectivity radius and the full curvature tensor. This way, our estimates can be easily compared to that of Cheng, Li and Yau, stated above for the case in which the injectivity radius of $M$ is positive.

\begin{corollary}\label{cor:hesscor}
In addition to the assumptions of Theorem \ref{th:derivativeestiamte} suppose also that the first covariant derivative of the curvature tensor is bounded. Then for all $T>0$ there exists a constant $C>0$ such that
\begin{equation}
\|\Hess \log p^M_t(\cdot,N)_x\| \leq C\left(\frac{1}{t}+\frac{\dimn}{t}\log\frac{1}{t}+\frac{d^2(x,N)}{t^2}\right)
\end{equation}
for all $x \in M$ and $t \in \left(0,T\right]$.
\end{corollary}

\begin{proof}
It was proved by Stroock in \cite{MR1715265} that for any continuous positive function $f$ there exists $C>0$ such that
\begin{equation}
t\frac{\|\Hess (P_t f)_x\|}{P_tf(x)} \leq C\left(1+h_t(x;f)\right)\nonumber
\end{equation}
for all $x \in M$ and $t\in \left(0,T\right]$ where $h_t(x;f)$ is defined by \eqref{eq:htdefn}. Choosing $f=p^M_t(\cdot,N)$, using the lower bound \eqref{eq:lbforder} and the on-diagonal upper bound \eqref{eq:ubforder} we obtain the corollary, by Theorem \ref{th:derivativeestiamte} and the fact that
\begin{equation}
\Hess \log P_tf = \frac{\Hess P_t f}{P_tf} - d \log P_t f \otimes d \log P_t f
\end{equation}
for all $t >0$.
\end{proof}

Note that the constant $C$ appearing in this estimate depends only on the injectivity radius of $M$, the length of the time interval $T$, the dimensions $\dimm$ and $\dimn$ and the bounds on the curvature tensor and its derivative.

\subsection{Further Applications}\label{ss:furtherapps}

The results in this article, in particular the estimate on the Hessian, will lead to a study of the space of continuous paths which end on a submanifold, potentially shedding light on the relationship between the geometry of the path space, the intrinsic geometry of the ambient manifold and the extrinsic geometry of the submanifold, extending recent work of Naber \cite{Naber} to the submanifold setting. A related application, which also concerns the extrinsic geometry of the submanifold, is the study of mean curvature flow (MCF). MCF typically concerns the codimension one $\dimn = \dimm -1$ case, where the normal component of the velocity of a point on a hypersurface $N$ is given by the mean curvature at that point. An important object here is the $F$-functional
\begin{equation}
F_{x,t}(N) := (2 \pi t)^{\frac{1}{2}} p_t^{M}(x,N)
\end{equation}
introduced by Huisken \cite{MR1030675} (and for which Theorem \ref{th:ihkform} now provides a probabilistic formula). The monotonicity of this functional under MCF is the reason for its importance. The \textit{entropy} $\lambda(N)$ of the hypersurface (see \cite[Section~3]{MR3312634}) is then defined by $\lambda(N) := \sup_{x,t} F_{x,t}(N)$. The entropy captures something of the extrinsic complexity of $N$ and determines which singularities can form as a result of the flow (and by Theorem \ref{th:quadcurvlb}, the entropy is always greater than or equal to $1$). According to Subsection \ref{ss:hltcomp}, one can interpret the $F$-functional as the ratio
\begin{equation}
F_{x,t}(N) = \frac{\frac{d}{dt}\mathbb{E}[L^N_t(X(x))]}{\frac{d}{dt}\mathbb{E}[L^{\mathbb{R}^{\dimm-1}}_t(X(0))]}.
\end{equation}
Indeed, the rate at which the expected local time increases for a hypersurface with high mean curvature ought to be greater than that for, say, a minimal hypersurface. So MCF has an interpretation in terms of hypersurface local time which could lead to further applications of our results.

Less abstractly, note that any non-degenerate It\^{o} diffusion in $\mathbb{R}^\dimm$, constrained to arrive in a submanifold at a fixed positive time, such as on a level set of a function or on the boundary of a domain, is a process of the type considered in this article (any non-degenerate diffusion on a smooth manifold induces a Riemannian metric, via the principle symbol of its generator, with respect to which it is a Brownian motion with drift). Multi-dimensional It\^{o} diffusions constrained by the value of an observable at a fixed future time could have applications in finance. Alternatively, consider the case of a particle diffusing inside a biological cell and which is known to exit the cell, via the membrane, at a particular time. Then the possible trajectories of the particle are distributed as those of the corresponding conditioned diffusion (and the metric could even be used to capture some of the internal structure of the cell).

Finally, a connection to statistics is given by the problem of manifold denoising (see \cite{NIPS2006_2997}), in which data is drawn from sampling a probability distribution supported around a submanifold of Euclidean space. The idea is then to reconstruct the submanifold (or derive certain topological properties of it, as in \cite{MR2383768}). The integrated heat kernel $p^{\mathbb{R}^\dimm}_{\sigma^2}(x,N)$ (normalized by the hypersurface volume $\vol_N(N)$) yields a probability distribution corresponding to the case in which observations of the submanifold are made subject to a Gaussian error.

\end{document}